\documentclass{article}

\newcommand{\Addresses}{{
		\bigskip
		\footnotesize
		
		\textsc{Department of Mathematics, Technion - Israel Institute of Technology, Haifa, Israel}\par\nopagebreak
		\textit{E-mail address:} \texttt{ofir.gor@technion.ac.il}
}}
\author{Ofir Gorodetsky} 
\date{}
\title{Statistics of entries of matrices from the general linear group over a finite field}

\usepackage[margin=1in]{geometry}
\usepackage[all]{xy}
\usepackage{amssymb}
\usepackage{amsmath}
\usepackage{amsthm}
\usepackage{hyperref}
\usepackage{url}

\theoremstyle{plain}
\newtheorem{thm}{Theorem}[section]
\newtheorem{lem}[thm]{Lemma}  
\newtheorem{prop}[thm]{Proposition}
\newtheorem{cor}[thm]{Corollary}
\newtheorem{prob}{Problem}

\theoremstyle{remark}
\newtheorem{remark}{Remark}

\newcommand{\FF}{\mathbb{F}}
\newcommand{\CC}{\mathbb{C}}

\newcommand{\ZZ}{\mathbb{Z}}
\newcommand{\PP}{\mathbb{P}}
\newcommand{\RR}{\mathbb{R}}
\newcommand{\OO}{\mathbb{O}}

\newcommand{\GLn}{\mathrm{GL}_n(\mathbb{F}_q)}
\newcommand{\GLnp}{\mathrm{GL}_n(\mathbb{F}_p)}
\newcommand{\Mn}{\mathrm{M}_n(\mathbb{F}_q)}
\newcommand{\rank}{\mathrm{rank}}
\newcommand{\Tr}{\mathrm{Tr}}
\newcommand{\Sp}{\mathrm{span}}
\newcommand{\TV}{\mathrm{TV}}

\numberwithin{equation}{section}

\begin{document}

\maketitle
\begin{abstract}
Let $\GLn$ be the group of $n\times n$ invertible matrices over a finite field $\FF_q$. How many matrices in $\GLn$ are there with a subset of their entries prescribed? Motivated by the random matrix theory of classical compact groups, we address this question and various variants, with a focus on asymptotic results that are valid as $n$ tends to infinity. 

Our methods are Fourier-analytic, with the main input being an exponential sum over $\GLn$, evaluated in 1955 by Hodges. Our results turn out to be related to investigations carried out in other fields, including coding theory, universality for $p$-adic matrices and algebraic combinatorics.
\end{abstract}
\section{Introduction}
Let $q$ be a prime power and let $\FF_q$ be the finite field with $q$ elements. Let $\GLn$ be the group of $n\times n$ invertible matrices over $\FF_q$. Let $L_1,\ldots,L_k \colon \Mn \to \FF_q$ be $k$ linearly independent linear functionals, where $\Mn$ is the vector space of $n \times n$ matrices over $\FF_q$. Let $I_1,\ldots,I_k$ be nonempty subsets of $\FF_q$.  In this work we seek to estimate
\begin{equation}\label{eq:mainquantity}
S_{\GLn}((L_i)_{i=1}^{k}, (I_i)_{i=1}^{k}) := \sum_{g \in \GLn} \prod_{i=1}^{k} \mathbf{1}_{L_i(g) \in I_i}
\end{equation}
in the limit as $n$ goes to $\infty$. We shall typically omit the subscript $\GLn$ in $S_{\GLn}$. All implied constants in this paper are absolute. We first observe that an upper bound of the expected order is trivial:
\begin{equation}\label{eq:trivial}
S((L_i)_{i=1}^{k}, (I_i)_{i=1}^{k}) \le \sum_{g \in \Mn} \prod_{i=1}^{k} \mathbf{1}_{L_i(g) \in I_i} = q^{-k}|\Mn| \prod_{i=1}^{k}|I_i|\asymp q^{-k}|\GLn|\prod_{i=1}^{k}|I_i|
\end{equation}
since $|\GLn| = \prod_{i=0}^{n-1} (q^n-q^i)=q^{n^2}\prod_{i=0}^{n-1} (1-q^{i-n}) \asymp q^{n^2}=|\Mn|$.
\subsection{General results}
Our first few results do not make assumptions on the functionals $(L_i)_{i=1}^{k}$ aside from linear independence. The next two theorems concern the case where $(I_i)_{i=1}^{k}$ are singletons.
\begin{thm}\label{thm:basic}
Let $1 \le k \le n^2$. Let 
 $L_1,\ldots,L_k\colon\Mn\to \FF_q$ be linearly independent linear functionals and $x_1,\ldots,x_k \in \FF_q$. We have
\begin{equation}\label{eq:nminus1}
S((L_i)_{i=1}^{k}, (\{x_i\})_{i=1}^{k}) = q^{-k} |\GLn| (1+O(q^{k-n})).
\end{equation}
\end{thm}
Theorem~\ref{thm:basic} gives an asymptotic count for the number of invertible matrices with $k$ fixed linear functionals, as long as $n-k \to \infty$. We may interpret \eqref{eq:nminus1} probabilistically. Let $g=g_n$ be a matrix chosen uniformly at random from $\GLn$. Let $(U_1,\ldots,U_k)$ be $k$ i.i.d.~random variables uniformly distributed on $\FF_q$. Then \eqref{eq:nminus1} is the same as
\begin{equation}\label{eq:prob}
	\PP( L_i(g)=x_i\, \forall 1 \le i \le k) -\PP( U_i=x_i\, \forall 1 \le i \le k) \ll q^{-n}.
\end{equation}
Our next result is a \textit{total variation} bound. Recall that given two  random variables $X$ and $Y$ taking values in a finite set $\Omega$, the total variation distance between $X$ and $Y$ is
\begin{equation}\label{eq:dtv}
d_{\TV}(X,Y) := \sup_{A \subseteq \Omega} | \PP(X \in A) -\PP(Y\in A)|=\frac{1}{2}\sum_{x \in \Omega} | \PP(X=x) - \PP(Y=x)|
\end{equation}
where the second equality is standard \cite[Eq.~(3.7)]{ABR}. Summing \eqref{eq:prob} over every choice of $(x_i)_{i=1}^{k}\in \FF_q^k$ we immediately deduce
\[ d_{\TV}( (L_i(g))_{i=1}^{k},(U_i)_{i=1}^{k})\ll q^{k-n}. \]
However, one can obtain a better bound, as shown in the following theorem.
\begin{thm}\label{thm:dtv}
Let $g=g_n$ be a matrix chosen uniformly at random from $\GLn$. Let $1\le k \le 2n$.  Let $L_1,\ldots,L_k\colon \Mn\to \FF_q$ be linearly independent linear functionals and $(U_i)_{i=1}^{k}$ be i.i.d.~random variables uniformly distributed on $\FF_q$. We have
\[ d_{\TV}( (L_i(g))_{i=1}^{k},(U_i)_{i=1}^{k})\ll q^{\frac{k}{2}-n}. \]
\end{thm}
More involved results on  $d_{\TV}( (L_i(g))_{i=1}^{k},(U_i)_{i=1}^{k})$ are stated and proved in \S\ref{sec:tv}.

The following is a variant of Theorem~\ref{thm:basic}, in the case where $q$ is a prime and $(I_i)_{i=1}^{k}$ are intervals.
\begin{thm}\label{thm:pv}
Let $1\le k \le n^2$. Suppose $q$ is a prime. Given integers $(x_i)_{i=1}^{k}$ together with positive integers $(h_i)_{i=1}^{k}$ less than $q$, let $I_i=\{ x_i + t \mod q: 0 \le t < h_i\}$. Let $L_1,\ldots,L_k\colon\Mn\to \FF_q$ be linearly independent linear functionals. We have
\[ S((L_i)_{i=1}^{k}, (I_i)_{i=1}^{k}) = \prod_{i=1}^{k} \frac{h_i}{q}|\GLn| \bigg(1+O\bigg( \frac{\prod_{i=1}^{k}\min\{\sqrt{h_i},\log (h_i)+3\}}{\prod_{i=1}^{k}h_i}q^{k-n}\bigg)\bigg). \]
\end{thm}
The error term in Theorem~\ref{thm:pv} goes to $0$ if e.g.~$h_i\ge \omega^{\frac{1}{k}}q^{1-\frac{n}{k}}(\log(q)+3)$ for all $i$, where $\omega=\omega(n,k,q)\to \infty$. As opposed to Theorem~\ref{thm:basic} where $k$ had to be less than $n$ for an asymptotic, here we can take $k\sim Cn$ for a given $C>1$, as long as the lengths $h_i$ are slightly larger than $q^{1-C^{-1}}$. We may even take $k$ to grow faster than linear in $n$ (as large as $k=n^2$), at the cost of restricting the meaningful range of the other parameters.

Theorem~\ref{thm:pv} is the first result of its kind where $n$ is allowed to grow. Previously,  Ahmadi and Shparlinski \cite[Theorem~9]{Ahmadi}, Hu and Li \cite[Theorem~1.5]{Equi} and Perret-Gentil \cite[Theorem~1.7]{PerretGentil} considered $S((L_i)_{i=1}^{k},(I_i)_{i=1}^{k})$ for $q$ and $I_i$ as in 
Theorem~\ref{thm:pv} in the regime where $q$ tends to $\infty$ while $n$ stays \textit{fixed}, and they had an additional constraint: the $k$ functionals  $(L_i(g))_{i=1}^{k}$ were $(g_{i,j})_{1 \le i,j\le n}$. See \S\ref{sec:largep} for more on the literature, and Theorem~\ref{thm:pvfunc} for an improvement on Theorem~\ref{thm:pv} when $(L_i(g))_{i=1}^{k}=(g_{i,j})_{1 \le i,j\le n}$.

There are cases where Theorem~\ref{thm:basic} is sharp, e.g.~when $L_i(g)= g_{1,i}$ for $i=1,\ldots,k$ ($k\le n$), see \S\ref{sec:sub} below. In many cases, however, one can do better as we now explain.

Each functional $L_i$ can be written as $L_i(g) = \Tr(A_ig)$ for a unique $A_i \in \Mn$. Conversely, any $A_i \in \Mn$ gives rise to a unique functional $g\mapsto \Tr(A_ig)$. The linear independence of $(L_i)_{i=1}^{k}$ is equivalent to the linear independence of $(A_i)_{i=1}^{k}$. We use the notation $A_i$ throughout. As it turns out, the quantity $S((L_i)_{i=1}^{k},(I_i)_{i=1}^{k})$ is largely governed by the rank distribution of the linear combinations of $(A_i)_{i=1}^{k}$, that is, by the values 
\begin{equation}\label{eq:deff}
f(r):=|\{ (c_1,\ldots,c_k) \in \FF_q^k:  \rank(c_1A_1+\ldots+c_kA_k)=r\}|
\end{equation}
as $r$ varies over $1,2,\ldots,n$. Theorems \ref{thm:basic}, \ref{thm:dtv} and \ref{thm:pv} are special cases of the much more refined Lemmas \ref{lem:gensingle}, \ref{lem:dtv} and \ref{lem:gen}, respectively, which are stated in terms of $f$. Theorems \ref{thm:basic}--\ref{thm:pv} are deduced from these lemmas via the trivial bound $f(r) \le q^k$. For certain families of functionals, the function $f$ defined in \eqref{eq:deff} has exact formulas which lead to significantly better results on $S((L_i(g))_{i=1}^{k},(I_i)_{i=1}^{k})$. Next we state stronger results in three natural families of functionals:
\begin{itemize}
	\item Diagonal functionals: $(L_i(g))_{i=1}^{n} = (g_{i,i})_{i=1}^{n}$. 
	\item Submatrix functionals: $(L_i(g))_{i=1}^{m_1m_2} = (g_{i,j})_{1 \le i \le m_1,\, 1\le j \le m_2}$.
	\item A refinement of Theorem~\ref{thm:pv} when $(L_i(g))_{i=1}^{n^2}=(g_{i,j})_{1\le i,j\le n}$.
\end{itemize}
\subsection{Distribution of diagonal}
Recall the falling factorial notation $(x)_n=x(x-1)(x-2)\cdots (x-(n-1))$ with the convention $(x)_0=1$. 
\begin{thm}\label{thm:diag}
Given $x_1,\ldots,x_n \in \FF_q$, let $m = | \{ 1 \le i \le n: x_i = 0\}|$. We have
\begin{equation}\label{eq:Sgiiform}
S((g_{i,i})_{i=1}^{n}, (\{x_i\})_{i=1}^{n}) =q^{\binom{n}{2}-n}\sum_{r=0}^{n}  \binom{n}{r}\prod_{i=1}^{n-r}(q^i-1)  \sum_{j=0}^{\min\{m,r\}}(-q)^j \binom{r}{j} \frac{(m)_j}{(n)_j} .
\end{equation}
\end{thm}
The left-hand side of \eqref{eq:Sgiiform} counts the number of matrices in $\GLn$ with prescribed diagonal elements. In the special case $(x_1,\ldots,x_n)=\vec{0}$, Theorem~\ref{thm:diag} gives a new proof of a result of Lewis, Liu, Morales, Panova, Sam and Zhang \cite{Lewis}, discussed in \S\ref{sec:Fze} below, which answered a question of R.~Stanley. Explicitly, when $(x_1,\ldots,x_n)=\vec{0}$,  Theorem~\ref{thm:diag} becomes 
\[S((g_{i,i})_{i=1}^{n}, (\{0\})_{i=1}^{n}) =q^{\binom{n}{2}-n}(q-1)^n \sum_{r=0}^{n}(-1)^r \binom{n}{r} \prod_{i=1}^{n-r} \frac{q^i-1}{q-1}.\]
In the case that $x_i \neq 0$ holds for all $1 \le i \le n$, Theorem~\ref{thm:diag} becomes 
\[S((g_{i,i})_{i=1}^{n}, (\{x_i\})_{i=1}^{n}) =q^{\binom{n}{2}-n}\sum_{r=0}^{n} \binom{n}{r} \prod_{i=1}^{n-r} (q^i-1),\]
which is a new result.
\begin{remark}
The inner sum in the right-hand side of \eqref{eq:Sgiiform} is ${}_{2}F_{1}(-m,-r;-n;q)$ where ${}_{2}F_{1}$ is the ordinary hypergeometric function.
\end{remark}
From Theorem~\ref{thm:diag} we deduce an equidistribution result.
\begin{cor}\label{cor:diag}
Given $x_1,\ldots,x_n \in \FF_q$, let $m=|\{1\le i \le n: x_i =0\}|$. We have
\begin{equation}\label{eq:diagpw} S((g_{i,i})_{i=1}^{n}, (\{x_i\})_{i=1}^{n})= q^{-n}|\GLn| (1+O(q^{-n}(qm+n))).
\end{equation}
Moreover, if $g=g_n$ is a matrix chosen uniformly at random from $\GLn$ and $(U_i)_{i=1}^{n}$ are i.i.d.~random variables uniformly distributed on $\FF_q$, we have
\begin{equation}\label{eq:TVdiag} d_{\TV}( (g_{i,i})_{i=1}^{n},(U_i)_{i=1}^{n})\ll \frac{\min\{\sqrt{nq},n\}}{q^n}.
\end{equation}
\end{cor}
See Remark~\ref{rem:bin} for an improvement on \eqref{eq:diagpw} in the regime $m\sim n/q$.

\subsection{Distribution of submatrices}\label{sec:sub}
When the functionals give the entries of a submatrix, we have the following result.
\begin{thm}\label{thm:submatrices}
Let $1\le m_1,m_2 \le n$. Given elements $(x_{i,j})_{1\le i \le m_1, \, 1 \le j \le m_2}$ from $\FF_q$, we have
\begin{equation}\label{eq:entryptprop} S(  (g_{i,j})_{\substack{1\le i \le m_1 \\ 1 \le j \le m_2}}, (\{x_{i,j}\})_{\substack{1 \le i \le m_1 \\ 1 \le j \le m_2}})=q^{-m_1 m_2 }|\GLn| ( 1+O( q^{\min\{m_1+m_2-1-n,0\}})).
\end{equation}
\end{thm}
To study the total variation distance we first prove an elementary formula for the left-hand side of \eqref{eq:entryptprop}; this formula may also be used to give an alternative proof of Theorem~\ref{thm:submatrices}, see Remark~\ref{rem:alt}.
\begin{lem}\label{lem:form}
Under the notation of Theorem~\ref{thm:submatrices}, assume without loss of generality that $m_1\le m_2$. Let $r$ be the rank of $(x_{i,j})_{1\le i\le m_1,\, 1\le j \le m_2} \in \mathrm{M}_{m_1 \times m_2}(\FF_q)$. Then
	\begin{equation}\label{eq:entryptprop2} S(  (g_{i,j})_{\substack{1\le i \le m_1 \\ 1 \le j \le m_2}}, (\{x_{i,j}\})_{\substack{1 \le i \le m_1 \\ 1 \le j \le m_2}})=q^{(n-m_1)r}\prod_{i=0}^{m_2-r-1}(q^{n-m_1}-q^i)\prod_{i=m_2}^{n-1}(q^n-q^i)
	\end{equation}
	if $r\ge  m_1+m_2-n$, and otherwise
		\begin{equation}\label{eq:entryptprop3} S(  (g_{i,j})_{\substack{1\le i \le m_1 \\ 1 \le j \le m_2}}, (\{x_{i,j}\})_{\substack{1 \le i \le m_1 \\ 1 \le j \le m_2}})=0.
	\end{equation}
\end{lem}
We use Lemma~\ref{lem:form} to prove the following.
\begin{thm}\label{thm:submatricestv}
	Let $1\le m_1,m_2 \le n$. If $g=g_n$ is a matrix chosen uniformly at random from $\GLn$ and $(U_{i,j})_{1\le i \le m_1,\, 1 \le j \le m_2}$ are i.i.d.~random variables uniformly distributed on $\FF_q$, we have
	\begin{equation}\label{eq:TVsub}
		d_{\TV}( (g_{i,j})_{\substack{1\le i\le m_1 \\ 1 \le j \le m_2}},(U_{i,j})_{\substack{1 \le i \le m_1 \\ 1 \le j \le m_2}})\asymp q^{\min\{m_1,m_2\}-1-n}. 
	\end{equation}
\end{thm}
\subsection{The all-entries case}\label{sec:allentries}
The following theorem improves on Theorem~\ref{thm:pv} when $(L_i(g))_{i=1}^{n^2}=(g_{i,j})_{1\le i,j\le n}$ and $h_i=h$.
\begin{thm}\label{thm:pvfunc}
Suppose $q$ is a prime. Given integers $(x_{i,j})_{1\le i,j\le n}$ and a positive integer $1\le h <q$, let $I_{i,j}=\{ x_{i,j} + t \mod q: 0 \le t < h\}$. For $n\ge 2$, we have
	\[ S((g_{i,j})_{1\le i,j\le n}, (I_{i,j})_{1\le i,j\le n}) = \left(\frac{h}{q}\right)^{n^2}|\GLn| \bigg(1+O\bigg( \bigg(3\frac{q^{\frac{n-1}{2n-1}}(\log(h)+3)}{h}\bigg)^{2n-1}\bigg)\bigg). \]
\end{thm}
The error term goes to $0$ as long as $h/(q^{\frac{n-1}{2n-1}}(\log (q)+3)) \to \infty$. 

In forthcoming work with Zelingher \cite{Zelingher2026}, we adapt the approach behind Theorems~\ref{thm:pv} and \ref{thm:pvfunc} to count matrices $g\in C$ that satisfy restrictions of the form $L_i(g) \in I_i$, where $C\subseteq \Mn$ is some conjugation-invariant subset.
\subsection{Past works}
\subsubsection{\texorpdfstring{$k=1$}{k=1}}
The literature on $k=1$ is somewhat disorganized, because some authors studied $S(L,I)$ itself, while others studied an exponential sum whose evaluation is equivalent to evaluating $S(L,I)$. Moreover, most authors were not aware of an old paper of Hodges \cite{Hodges} that seems to be the first to evaluate these exponential sums.

Buckhiester \cite{Buckhiester}, Bender \cite{Bender} and Prasad \cite[Lemma~2]{Prasad} computed $S((L_i)_{i=1}^{k},(I_i)_{i=1}^{k})$ for $k=1$, $L_1(g)=\Tr(g)$ and $|I_1|=1$, i.e., they counted matrices with prescribed trace (more generally, they considered matrices of any fixed rank with prescribed trace). This is also treated  in Stanley \cite[Proposition~1.10.15]{Stanley}. The case of $k=1$, $|I_1|=1$ and \textit{general} $L_1$ is considered by Stanley as well \cite[Exercise~1.196 and its solution]{Stanley}. Let $\psi\colon \FF_q\to \CC^{\times}$ be a nontrivial additive character. Kim \cite{Kim}, Fulman \cite{Fulman} and Chae and Kim \cite{CK} evaluated $\sum_{g \in \GLn} \psi( \Tr(g))$. Li and Hu \cite[Theorem~2.1]{LiHu} proved that, for every $A \in \Mn$,
\begin{equation}\label{eq:Hodges} \sum_{g \in \GLn} \psi( \Tr(Ag)) = (-1)^{\rank(A)} q^{\binom{n}{2}} \prod_{i=1}^{n-\rank(A)}(q^i-1).
\end{equation}
 They also considered sums over $\mathrm{SL}_n(\FF_q)$. As we shall see in Remark~\ref{rem:equiv}, formula \eqref{eq:Hodges} is \textit{equivalent} to determining $S(L,I)$. The aforementioned results on $S(L,I)$ and $\sum_{g \in \GLn} \psi( \Tr(Ag))$ are preceded by a 1955 paper of Hodges \cite{Hodges}. Theorem~7 of \cite{Hodges} evaluates the exponential sum $\sum_{g \in \mathrm{M}_{n\times m}:\, \rank(g)=r} \psi(\Tr(Ag))$ for any $0\le r \le \min\{n,m\}$ and any $A\in \mathrm{M}_{m\times n}(\FF_q)$, and in particular recovers \eqref{eq:Hodges}. His proof is completely elementary. 
 
 The sum he considered was also implicitly evaluated in the coding theory community by Delsarte \cite{Delsarte}, Ravagnani \cite{Ravagnani,Ravagnani2018} and Beelen and  Ghorpade \cite{Beelen}, see the discussion in Balasubramanian, Kaipa and Khurana \cite{Balasubramanian} who provide yet another evaluation.

\subsubsection{Large-\texorpdfstring{$p$}{p} limit: algebraic geometry}\label{sec:largep}
Suppose $q=p$ is a prime. Hu and Li \cite[Theorem~1.5]{Equi} proved a result which in particular shows that 
\begin{equation}\label{eq:huli}
S( (g_{i,j})_{1\le i,j\le n}, (I_{i,j})_{1\le i,j\le n}) \sim \prod_{1 \le i,j \le n} |I_{i,j}| \frac{|\GLnp|}{p^{n^2}} \end{equation}
in the limit as $p\to \infty$ for intervals $I_{i,j}$ of size \textit{proportional $p$}. Among other problems, Perret-Gentil \cite[Theorem~1.7]{PerretGentil} studied a quantitative version of \eqref{eq:huli} (including for other classical groups). His result, in the case of $\GLnp$, says that 
\begin{equation}\label{eq:pg}
S( (g_{i,j})_{1\le i,j\le n}, (I_{i,j})_{1\le i,j\le n}) = \prod_{1 \le i,j \le n} |I_{i,j}| \frac{|\GLnp|}{p^{n^2}} + O_n(p^{n^2-\frac{1}{2}} (\log p)^{2})
\end{equation}
for intervals $I_{i,j}$, which improves \eqref{eq:huli} but is strictly weaker than our Theorem~\ref{thm:pv}. Perret-Gentil relied on Deligne's work on the Weil conjectures. 

Next we assume $n \ge 2$ and identify $\FF_p$ with $\{-(p-1)/2,\ldots,(p-1)/2\}\subseteq \ZZ$. Given $1 \le H\le (p-1)/2$, Ahmadi and Shparlinski \cite[Theorem~9]{Ahmadi} considered the number of $n\times n$ matrices with rank at most $k$ (when reduced modulo $p$) whose entries lie in $\{-H,\ldots,H\}$. While this is a seemingly different problem, if one makes the choice $k=n-1$ and considers the complement of the set they studied, one obtains an estimate for invertible matrices, namely 
\begin{equation}\label{eq:shpar}
S( (g_{i,j})_{1\le i,j\le n}, (\{-H,\ldots,H\})_{1\le i,j\le n}) = (2H+1)^{n^2} \big(1-\frac{1}{p}\big) + O_n\big(p^{\frac{n^2-1}{2}}(\log p)^{n^2}+H^{n^2-2}p^{\frac{1}{2}}(\log p)^2\big).
\end{equation}
This is meaningful for $H/(p^{\frac{1}{2}-\frac{1}{2n^2}}\log p)\to \infty$ and fixed $n$. The main input into \cite[Theorem~9]{Ahmadi} is a general result of Fouvry \cite{Fouvry}, relying on work of Katz and Laumon. When specialized to $I_{i,j}=\{-H,\ldots,H\}$, Theorem~\ref{thm:pvfunc} is superior to \eqref{eq:shpar}, in addition to allowing $n$ to vary. 
The dependence on $n$ in \eqref{eq:shpar} cannot be removed in general, because the `correct' main term in the large-$n$ limit should be $(2H+1)^{n^2} |\GLnp|/p^{n^2}$ in place of $(2H+1)^{n^2}(1-1/p)$.
\subsubsection{Large-\texorpdfstring{$n$}{n} limit: universality}
Let $\mu$ be a probability measure on $\FF_q$, not concentrated on any proper affine subspace (when regarding $\FF_q$ as an $\FF_p$-vector space, when $p$ is the characteristic). Let $M_n\in \Mn$ be a random matrix with i.i.d.~entries distributed according to $\mu$. Charlap, Rees and Robbins \cite{Charlap} proved that
\begin{equation}\label{eq:CRR}
\lim_{n\to \infty} \PP(M_n \in \GLn) = \prod_{i=1}^{\infty}\left(1-\frac{1}{q^i}\right).
\end{equation}
Maples \cite{Maples} (cf.~Nguyen--Paquette \cite{Nguyen}) proved a quantitative version: for an absolute constant $c>0$,
\[  \PP(M_n \in \GLn) = \prod_{i=1}^{\infty}\left(1-\frac{1}{q^i}\right) + O( e^{-c\alpha n})\]
holds, where $\alpha:=1-\max_{A}\mu(A)$ and the maximum is taken over proper affine subspaces $A\subseteq \FF_q$. When $q$ is prime and the entries of $M_n$ are independent but not identically distributed, a special case of a result of Wood  \cite[Thm.~1.2]{Wood} shows that \eqref{eq:CRR} persist; her result is qualitative (no rate of convergence is given).

Taking $q$ to be a prime, these results can be contrasted with our Theorem~\ref{thm:pvfunc}. We consider $\PP(M_n\in \GLn)$ where the entries of $M_n$ are independent, but not necessarily identically distributed: each is selected uniformly at random from a different interval of length $h$; this falls under Wood's framework. Our rate of convergence is $q^{-nc}$ as soon as $h$ is slightly larger than $\sqrt{q}$ (namely $h\ge 3 (\log (q)+3)q^{\frac{1+c}{2}}$).
The only work we are aware of that provides a comparable error is the recent paper of Koenig and Nguyen \cite{Koenig}, which requires $h$ to be proportional to $p$.

\subsubsection{Fixed zero entries and algebraic combinatorics}\label{sec:Fze}
Various authors \cite{Haglund,Stembridge,Klein,Lewis2,Cotardo,ChenSelover} considered $S((L_i)_{i=1}^{k}, (I_i)_{i=1}^{k})$ with $I_1=\ldots=I_k = \{0\}$ and $L_i(g) = g_{x_i,y_i}$ for special choices of $(x_i,y_i) \in \{1,\ldots,n\}\times \{1,\ldots,n\}$. In words, they counted elements of $\GLn$ with some of their entries prescribed to be $0$. For simplicity of notation, given $S \subseteq \{1,\ldots,n\}\times \{1,2,\ldots,n\}$ we write $T_q(n,S)$ for the number of matrices $g\in \GLn$ with $g_{i,j}=0$ for every $(i,j)\in S$.

Haglund \cite[Theorem~1]{Haglund} computed $T_q(n,S)$ when $S$ is the \textit{complement} of a Ferrers board $F$ (so he counted invertible matrices supported on $F$). Recall that a Ferrers board is a subset of $\{1,\ldots,n\} \times \{1,\ldots,n\}$ such that $(i,j)\in F \implies (k,p)\in F$ for all $1\le k \le i$ and $j\le p \le n$. A consequence of Haglund's result is that $T_q(n,S)$ is a polynomial in $q$ in this case. This polynomial may be $0$; he showed that it is not identically $0$ if and only if one can place $n$ nonattacking rooks on the squares of $F$ (this means no two rooks are in the same column or row).

For $n=7$ Stembridge \cite{Stembridge} provided an example of $S$ for which $T_q(7,S)$ is not a polynomial in $q$ (this was motivated by a conjecture of Kontsevich).

Lewis, Liu, Morales, Panova, Sam and Zhang \cite{Lewis} considered $S=\{ (i,i): 1 \le i \le n\}$, i.e., they counted invertible matrices with zeros on the diagonal, answering a question of R.~Stanley (cf.~Exercise 1.197 in \cite{Stanley}). They obtained a closed form formula in this case. This quantity may be viewed as a $q$-analogue of the number of derangements (permutations without fixed points), because derangements may be characterized as permutations whose associated permutation matrix has only zeros on its diagonal. More generally, in \cite{Lewis} the authors also computed the number of matrices of fixed rank $r$ ($r\in \{1,2,\ldots,n\}$) with only zeros on the diagonal.

For further results on formulas for $T_q(n,S)$ for special sets $S$ see \cite{Klein,Lewis2,Cotardo}. Most of these works, as well as \cite{Lewis}, also consider invertible \textit{symmetric} matrices with fixed zero entries.

A general formula for $T_q(n,S)$ was proved by Ravagnani  \cite[Rem.~65]{Ravagnani} (cf.~Lewis and Morales \cite[Cor.~2.2]{Lewis2}).
\begin{lem}[Ravagnani]\label{lem:genzero}
	Let $S \subseteq \{1,2,\ldots,n\}\times\{1,2,\ldots,n\}$. Given $s\in S$ we write $E_s \in \Mn$ for the elementary matrix with $1$ in the $s$-th entry and $0$ otherwise. We have
	\[   T_q(n,S) = q^{\binom{n}{2}-|S|} \sum_{r=0}^{n} (-1)^r \prod_{i=1}^{n-r}(q^i-1) f(r)\]
	for 
	\[ f(r) := |\{ (a_s)_{s \in S}  \in \FF_q^{S}: \rank\big(\sum_{s \in S} a_s E_s\big)=r\}|.\]
\end{lem}
Both Ravagnani and Lewis--Morales obtain the formula from the MacWilliams identities, as derived by Delsarte in the context of rank-metric codes \cite{Delsarte}. In \S\ref{sec:iden} we give a direct proof using Hodges' work.

\subsection{Comparison with random matrix theory}\label{sec:classical}
Let us compare our results on $\GLn$ with those on $\OO(n)$, the (real) orthogonal group. We recommend Chapter 2 of E.~Meckes' book \cite{Meckes} as a source for the random matrix theory results. Throughout, $U_n$ is a matrix chosen randomly from $\OO(n)$ according to the probability Haar measure, and $g_n$ is chosen uniformly at random from $\GLn$.

In \cite[Lemma~2.1]{Meckes} it is shown that the entries of $U_n$ are identically distributed. This is the case for $\GLn$ as well, as follows from the fact that $g_n$ and $P g_nQ$ have the same distribution when $P,Q\in \GLn$ are given permutation matrices.

It is a classical fact that the first row of $U_n$ is distributed like a random uniform vector from the unit sphere $S^{n-1} \subseteq \RR^{n}$ \cite[Chapter~1.2]{Meckes}. In comparison, a simple counting argument shows that the first row of $g_n$ is distributed like a random uniform vector from $\FF_q^n \setminus \{ \vec{0}\}$.

The distribution of submatrices of $U_n$ has received considerable attention \cite{DF,DEL,Eaton,Jiang,JiangMa,Stewart}. 
The optimal result was proved independently by Jiang and Ma \cite{JiangMa} and Stewart \cite{Stewart}. The result says the following: for each $n\ge 1$, let $U_n \in \OO(n)$ be a random orthogonal matrix. Suppose that $p_n$ and $q_n$ are positive integers with $p_n q_n = o(n)$. Let $U_n(p_n,q_n)$ denote the top-left $p_n\times q_n$ block of $U_n$. Let $Z(p_n,q_n)$ denote a $p_n \times q_n$ random  matrix of i.i.d.~standard Gaussians. Then the total variation distance
\[ d_{\TV}(\sqrt{n} U_n(p_n,q_n), Z(p_n,q_n))\]
goes to $0$ as $n \to \infty$. Jiang and Ma, in the aforementioned work, also demonstrated that $p_nq_n=o(n)$ cannot be improved. This should be compared to Theorem~\ref{thm:submatricestv} for $\GLn$, which studies a similar question but has the significantly more relaxed condition $n-\min\{p_n,q_n\} \to \infty$.

Let $A_1,\ldots,A_k$ be $n\times n$ real matrices with $\Tr(A_i A_j^T)=n\delta_{ij}$. Chatterjee and Meckes \cite[Thm.~11]{Chatterjee} bounded the $L_1$-Wasserstein distance between $(\Tr(A_i U_n))_{i=1}^{k}$ and a vector of $k$ independent standard Gaussians. They proved that the distance goes to $0$ as $k=o(n)$ (see \cite{DAristotile,MeckesLinear,Keating} for results on the $k=1$ case and \cite{Collins} for another approach when $k$ is fixed). This can be seen as an analogue of our Theorem~\ref{thm:dtv}, where convergence holds in the wider range $2n-k\to \infty$. In fact, for fixed $q$ we determine in \S\ref{sec:tv} the exact range in which the distance in Theorem~\ref{thm:dtv} tends to $0$: it turns out to be $(n^2-k)/n\to \infty$.
\begin{remark}
The range of Chatterjee--Meckes is sharp, as can be seen by considering $A_i = \sqrt{n} E_{1,i}$ in which case $(\Tr(A_i U_n))_{i=1}^{k}$ is distributed like the first $k$ coordinates of a point uniformly distributed on the $(n-1)$-sphere of radius $\sqrt{n}$. When $k/n$ tends to a positive constant, the $L^1$-Wasserstein distance does not go to $0$; we leave the details to the reader.
\end{remark}
\begin{remark}
If one specializes Chatterjee--Meckes to $(A_i)_{i=1}^{k} = (\sqrt{n}E_{i,j})_{1 \le i\le p_n,\, 1\le j \le q_n}$, then their result implies that the $L^1$-Wasserstein distance between $\sqrt{n} U_n(p_n,q_n)$ and $Z(p_n,q_n)$ goes to $0$ as $p_n q_n=o(n)$. So their result recovers the Wasserstein analogue of the aforementioned total variation results.
\end{remark}
\subsection*{Acknowledgments}
O.G.~is supported by the Israel Science Foundation (grant no.~2088/24), an Alon Fellowship and the Rabbi Dr.~Roger Herst Faculty Fellowship.
\subsection*{AI usage}
A draft of this paper was prepared, and was uploaded to an LLM upon completion. Apart from proofreading and helping navigate some of the coding theory and universality literature, the LLM ended up contributing the argument behind Theorem~\ref{thm:pvfunc} (namely the proof of \eqref{eq:rankr}) and the results in \S\ref{sec:tv} on total variation.

\section{Strategy}\label{sec:strategy}
Let $(L_i)_{i=1}^{k}$ and $(I_i)_{i=1}^{k}$ be as in the introduction. Recall the notation $A_i\in \Mn$ for the unique matrix with $L_i(g) = \Tr(A_ig)$. We begin by giving a Fourier-analytic formula for $S((L_i)_{i=1}^{k}, (I_i)_{i=1}^{k})$ which will serve as the basis for our proofs. Fix a nontrivial additive character $\psi \colon \FF_q\to \CC^{\times}$. Given a subset $I_i \subseteq \FF_q$, our initial step is expanding $t\mapsto \mathbf{1}_{t\in I_i}$ using additive characters. We express
\[\mathbf{1}_{t \in I_i} = q^{-1} \sum_{a \in \FF_q} \widehat{I_i}(-a) \psi(at)\]
for every $t \in \FF_q$, where for $X \subseteq \FF_q$ we define
\[ \widehat{X}(a) := \sum_{x \in X} \psi(ax).\]
In this way, our quantity \eqref{eq:mainquantity} transforms into
\begin{equation}\label{eq:Fourier}
S((L_i)_{i=1}^{k}, (I_i)_{i=1}^{k}) = q^{-k}\sum_{(a_1,\ldots,a_k) \in \FF^k_q} \prod_{i=1}^{k} \widehat{I_i}(-a_i) \sum_{g \in \GLn}\psi\big(\sum_{i=1}^{k} L_i(g) a_i\big).
\end{equation}
The main term $ \prod_{i=1}^{k}(|I_i|/q)|\GLn|$ arises from the zeroth frequency $(a_1,\ldots,a_k)=\vec{0}$:
\begin{equation}\label{eq:zeroth}
q^{-k} \prod_{i=1}^{k} \widehat{I_i}(-0)\sum_{g \in \GLn}\psi\big(\sum_{i=1}^{k} L_i(g) 0\big) = q^{-k} |\GLn| \prod_{i=1}^{k}|I_i|.
\end{equation}
We input Hodges' formula \eqref{eq:Hodges} into \eqref{eq:Fourier}, obtaining
\begin{equation}\label{eq:Fourierinput}
S((L_i)_{i=1}^{k}, (I_i)_{i=1}^{k}) = q^{\binom{n}{2}-k}\sum_{ (a_1,\ldots,a_k) \in \FF^k_q} \prod_{i=1}^{k} \widehat{I_i}(-a_i) (-1)^{\rank(\sum_{j=1}^{k} A_j a_j)}\prod_{i=1}^{n-\rank(\sum_{j=1}^{k} A_ja_j)}(q^i-1).
\end{equation}
The identity \eqref{eq:Fourierinput} is the basis for our results. We draw several conclusions from \eqref{eq:Fourierinput}.
\begin{lem}\label{lem:gensingle}
Let $q$ be a prime power. Let $L_1,\ldots,L_k \colon \Mn \to \FF_q$ be $k$ linearly independent linear functionals given by $L_i(g)=\Tr(A_ig)$, $A_i \in \Mn$. Let $(x_i)_{i=1}^{k}$ be a sequence of elements of $\FF_q$. We have
\[ S((L_i)_{i=1}^{k}, (\{x_i\})_{i=1}^{k}) =q^{-k}|\GLn| \bigg( 1+O\bigg( \sum_{r=1}^{n} f(r) q^{-nr+\binom{r}{2}}\bigg)\bigg) \]
where $f$ is defined in \eqref{eq:deff}.
\end{lem}
\begin{proof}
We apply \eqref{eq:Fourierinput} with $I_i = \{ x_i\}$. We observe $\widehat{I_i}(-a_i)=\psi(-a_ix_i)$ has absolute value $1$ and that $\widehat{I_i}(0)=1$. We group the summands in the right-hand side of \eqref{eq:Fourierinput} according to the value of $\rank(\sum_{i=1}^{k} a_i A_i)$ and separate the contribution of $r=0$ (which gives the main term) from $1 \le r \le n$. The result is simplified by observing that $q^{\binom{n}{2}}\prod_{i=1}^{n-r}(q^i-1) \ll |\GLn| q^{-nr+\binom{r}{2}}$.
\end{proof}
Let $g=g_n$ be chosen uniformly at random from $\GLn$. Lemma~\ref{lem:gensingle} can be expressed as follows.
\begin{cor}\label{cor:gensingle}
In the notation of Lemma~\ref{lem:gensingle}, if $(U_i)_{i=1}^{k}$ are i.i.d.~random variables uniformly distributed on $\FF_q$ then
\[  \PP( L_i(g)=x_i\, \forall 1\le i \le k) - \PP( U_i=x_i \, \forall 1 \le i \le k) \ll q^{-k} \sum_{r=1}^{n} f(r) q^{-nr+\binom{r}{2}}. \]
\end{cor}
From Corollary~\ref{cor:gensingle} summed over $\{x_i\}_{i=1}^{k} \in \FF_q^k$, we have
\[  d_{\TV}((L_i(g))_{i=1}^{k}, (U_i)_{i=1}^{k}) \ll  \sum_{r=1}^{n} f(r) q^{-nr+\binom{r}{2}}. \]
However, one can do better as shown in the next lemma.
\begin{lem}\label{lem:dtv}
	Let $q$ be a prime power. Let $L_1,\ldots,L_k \colon \Mn \to \FF_q$ be $k$ linearly independent linear functionals given by $L_i(g)=\Tr(A_ig)$, $A_i \in \Mn$. Let $(U_i)_{i=1}^{k}$ be i.i.d.~random variables uniformly distributed on $\FF_q$. We have
	\[  d_{\TV}((L_i(g))_{i=1}^{k}, (U_i)_{i=1}^{k}) \ll  \sqrt{\sum_{r=1}^{n} f(r) q^{2\left(-nr+\binom{r}{2}\right)}} \]
	where $f$ is defined in \eqref{eq:deff}.
\end{lem}
\begin{proof}
	We apply \eqref{eq:Fourierinput} with $I_i = \{ x_i\}$ to obtain
	\[	S((L_i)_{i=1}^{k}, (\{x_i\})_{i=1}^{k}) = q^{\binom{n}{2}-k}\sum_{r=0}^{n} (-1)^{r}\prod_{i=1}^{n-r}(q^i-1)  \sum_{\substack{ (a_1,\ldots,a_k) \in \FF^k_q \\ \rank(\sum_{j=1}^{k} a_jA_j)=r}} \psi(- \sum_{i=1}^{k} a_i x_i).\]
	Note that $\rank(\sum_{j=1}^{k}a_jA_j)=0$ if and only if $a_j=0$ for all $j$. Identifying the contribution of $r=0$ as $|\GLn|\PP( U_i=x_i \, \forall 1 \le i \le k)$, and dividing by $|\GLn|$, this becomes
	\begin{multline}\label{eq:PPdiff}  \PP( L_i(g)=x_i\, \forall 1\le i \le k) - \PP( U_i=x_i \, \forall 1 \le i \le k) \\=  q^{-k}\sum_{r=1}^{n} \frac{(-1)^{r}\prod_{i=1}^{n-r}(q^i-1) q^{\binom{n}{2}}}{|\GLn|} \sum_{\substack{ (a_1,\ldots,a_k) \in \FF^k_q \\ \rank(\sum_{j=1}^{k} a_jA_j)=r}} \psi(- \sum_{i=1}^{k} a_i x_i).
	\end{multline}
	In general, given a function $g\colon \FF_q^k \to \CC$ given by $g(\mathbf{x}) = \sum_{\mathbf{a}\in \FF_q^{k}} c_{\mathbf{a}} \psi(-\sum_{i=1}^{k}a_i x_i)$ for some complex constants $(c_{\mathbf{a}})_{\mathbf{a}\in \FF_q^k}$ (where $\mathbf{x}=(x_1,\ldots,x_k)$ and $\mathbf{a}=(a_1,\ldots,a_k)$), we have
	\[ \sum_{\mathbf{x}\in \FF_q^k} |g(\mathbf{x})|^2 =q^k\sum_{\mathbf{a}\in \FF_q^{k}} |c_{\mathbf{a}}|^2\]
	by orthogonality of $\mathbf{x}\mapsto \psi(-\sum_{i=1}^{k}a_i x_i)$ with respect to the inner product $\langle g_1,g_2\rangle = \sum_{\mathbf{x} \in \FF_q^k} g_1(x)\overline{g_2}(x)/q^k$. Hence, by applying Cauchy--Schwarz and using the Fourier expansion of $g(\mathbf{x}) = \PP( L_i(g)=x_i\, \forall 1\le i \le k) - \PP( U_i=x_i \, \forall 1 \le i \le k) $ given in \eqref{eq:PPdiff}, we get that
	\begin{align*}  
		&\sum_{\mathbf{x}\in \FF_q^k} |\PP( L_i(g)=x_i\, \forall 1\le i \le k) - \PP( U_i=x_i \, \forall 1 \le i \le k)| \\
		&\qquad\le q^{\frac{k}{2}} \sqrt{ \sum_{\mathbf{x}\in \FF_q^k} | \PP( L_i(g)=x_i\, \forall 1\le i \le k) - \PP( U_i=x_i \, \forall 1 \le i \le k)|^2}  \\
		&\qquad =  \sqrt{\sum_{r=1}^{n} \frac{\prod_{i=1}^{n-r}(q^i-1)^2 q^{2\binom{n}{2}}}{|\GLn|^2} f(r).}
	\end{align*}
	The bound is simplified by observing that $q^{\binom{n}{2}}\prod_{i=1}^{n-r}(q^i-1) \ll |\GLn| q^{-nr+\binom{r}{2}}$.
\end{proof}
The next lemma estimates  $S((L_i)_{i=1}^{k}, (I_i)_{i=1}^{k})$ for general sets $I_i$, but is most useful for structured sets such as intervals. It is a generalization of Lemma~\ref{lem:gensingle}.
\begin{lem}\label{lem:gen}
Let $q$ be a prime power. Let $L_1,\ldots,L_k \colon \Mn \to \FF_q$ be $k$ linearly independent linear functionals given by $L_i(g)=\Tr(A_ig)$, $A_i \in \Mn$. Let $I_1,\ldots,I_k$ be nonempty subsets of $\FF_q$. We have
\[ S((L_i)_{i=1}^{k}, (I_i)_{i=1}^{k}) = q^{-k}\big(\prod_{i=1}^{k}|I_i|\big) |\GLn| (1+O( M_1(\mathbf{L},\mathbf{I}))) \]
where $\mathbf{L}=(L_i)_{i=1}^{k}$, $\mathbf{I}=(I_i)_{i=1}^{k}$ and
\[ M_1(\mathbf{L},\mathbf{I}):= \sum_{r=1}^{n} q^{-nr+\binom{r}{2}}\sum_{\substack{ (a_1,\ldots,a_k) \in \FF^k_q\setminus \{\vec{0}\}\\ \rank(\sum_{j=1}^{k} a_jA_j)=r}} \prod_{i=1}^{k} \frac{|\widehat{I_i}(-a_i)|}{\widehat{I_i}(0)}.\]
\end{lem}
\begin{proof}
Recall $(A_i)_{i=1}^{k}$ are linearly independent. The contribution of the case $(a_1,\ldots,a_k)=\vec{0}$ to the right-hand side of \eqref{eq:Fourierinput} is exactly $q^{-k} (\prod_{i=1}^{k} |I_i| )|\GLn|$ by \eqref{eq:zeroth}. From \eqref{eq:Fourierinput} and the triangle inequality,
\[\big|S((L_i)_{i=1}^{k}, (I_i)_{i=1}^{k}) -q^{-k}\big(\prod_{i=1}^{k}|I_i|\big)|\GLn| \big| \le 
q^{\binom{n}{2}-k} \sum_{r=1}^{n}\prod_{i=1}^{n-r}(q^i-1) \sum_{\substack{(a_1,\ldots,a_k)\in \FF_q^k \setminus \{\vec{0}\}\\ \rank(\sum_{j=1}^{k} a_j A_j)=r}} \prod_{i=1}^{k} |\widehat{I_i}(-a_i)| .\]
By observing that $q^{\binom{n}{2}}\prod_{i=1}^{n-r}(q^i-1) \ll |\GLn| q^{-nr+\binom{r}{2}}$ and $\widehat{I_i}(0)=|I_i|$, the conclusion follows.
\end{proof}
\begin{remark}\label{rem:equiv}
If $k=1$, we can simplify \eqref{eq:Fourier} considerably. In this case we drop the indices and write $L$ and $I$ for $L_1$ and $I_1$:
\[S(L,I) = q^{-1}\sum_{a \in \FF_q}  \widehat{I}(-a) \sum_{g \in \GLn}\psi(L(g) a) = q^{-1}\sum_{a \in \FF_q}  \widehat{I}(-a) \sum_{g \in \GLn}\psi(L(ag)). \]
Given any $a \in \FF_q^{\times}$, $(ag)_{g\in \GLn}$ is just a permutation of $(g)_{g \in \GLn}$. This allows us to write
\begin{align*}
S(L,I) &= \frac{|I|}{q} |\GLn|+ q^{-1} \sum_{a \in \FF_q^{\times}} \widehat{I}(-a) \sum_{g \in \GLn} \psi(L(g))\\
&=\frac{|I|}{q} |\GLn| +  \big( \mathbf{1}_{0\in I} - \frac{|I|}{q}\big)\sum_{g \in \GLn} \psi(L(g))
\end{align*}
because $\sum_{a \in \FF_q^{\times}} \widehat{I}(-a) = q\mathbf{1}_{0 \in I}-|I|$, showing that $S(L,I)$ is determined by $\sum_{g \in \GLn}\psi(L(g))$. On the other hand, $\sum_{g \in \GLn}\psi(L(g))$ is determined by $S(L,\{0\})$. Indeed, by choosing $I = \left\{0\right\}$ we have
\[ \sum_{g \in \GLn} \psi(L(g)) = \big(1-\frac{1}{q}\big)^{-1}\big( S(L,\{0\}) - \frac{1}{q} |\GLn|\big).\]
\end{remark}
\subsection{Coding theory}
Let $X_1,\ldots,X_k\in \mathrm{M}_{n\times m}(\FF_q)$ be linearly independent matrices and let $C:=\Sp_{\FF_q}(X_1,\ldots,X_k)$. The space $C$, equipped with the metric $d(A,B)=\rank(A-B)$, is called a (linear) rank-metric code. The rank distribution of $C$ is the function $f(r)=f_C(r):=\#\{A \in C: \rank(A)=r\}$ which is central to our work.

Rank-metric codes were introduced by Delsarte \cite{Delsarte} and rediscovered by Gabidulin \cite{Gabidulin}. A distinguished quantity of $C$ is the \textit{minimum rank distance}, $d=d_C$, defined as the smallest positive integer such that $f_C(d)\neq 0$. Assuming $m\ge n$, the Singleton bound states $d\le n-\lceil \frac{k}{m}\rceil +1$ \cite[Thm.~5.4]{Delsarte}; codes attaining equality are called maximal-rank-distance (MRD) codes. The rank distribution of MRD codes is determined by $d,n,m,k,q$ alone \cite[Thm.~5.6]{Delsarte}.

As we explain in \S\ref{sec:tv}, a sharp upper bound on $f(1)$ is implicitly proved by Schaathun \cite{Schaathun}; we include in \S\ref{sec:tv} self-contained proofs for upper bounds on $f(r)$.
\section{Proofs}
\subsection{Proof of Theorem~\ref{thm:basic}}
By Lemma~\ref{lem:gensingle}, since $f(r)\le q^k$ for $1\le r \le n$, and since $-nr+\binom{r}{2}$ is decreasing for $r$ in the range $1\le r \le n$, the result follows. 
\subsection{Proof of Theorem~\ref{thm:dtv}}
By Lemma~\ref{lem:dtv}, since $f(r)\le q^k$ for $1\le r \le n$, and since $-nr+\binom{r}{2}$ is decreasing for $r$ in the range $1\le r \le n$, the result follows.
\subsection{Proof of Theorem~\ref{thm:pv}}
We suppose that $q$ is a prime. If $q=2$ then each $I_i$ is a singleton, in which case the result is a special case of Theorem~\ref{thm:basic}. From now on we suppose that $q>2$. We take $\psi(x) = e^{2\pi i x/q}$ for $x \in \FF_q$. Recall $I_i=\{ x_i + t \mod q: 0 \le t < h_i\}$. By Lemma~\ref{lem:gen}, since $f(r) \le q^k$ for $1\le r \le n$, and $-nr+\binom{r}{2}$ is decreasing in the range $1\le r \le n$, we have
\[ S((L_i)_{i=1}^{k}, (I_i)_{i=1}^{k}) = q^{-k}\big(\prod_{i=1}^{k}|I_i|\big) |\GLn| (1+O( q^{k-n}M_2(\mathbf{I}))) \]
where
\[ M_2(\mathbf{I}):= q^{-k}\sum_{ (a_1,\ldots,a_k) \in \FF^k_q\setminus \{\vec{0}\}} \prod_{i=1}^{k} \frac{|\widehat{I_i}(-a_i)|}{\widehat{I_i}(0)} \le \prod_{i=1}^{k} \Bigg( q^{-1}\sum_{a_i\in \FF_q}\frac{|\widehat{I_i}(-a_i)|}{\widehat{I_i}(0)}\Bigg).\]
We estimate the inner $a_i$-sum. For $a \in \FF_q^{\times}$,
\[\big| \widehat{I_i}(a)\big| = \big|\sum_{0 \le t < h_i} \psi(at)\big| = \bigg| \frac{\psi(ah_i)-1}{\psi(a)-1}\bigg| \le \frac{2}{|\psi(a)-1|} = \frac{1}{|\sin(a\pi/q)|} \le \frac{1}{2\| a/q\|}\]
where $\|t\|$ is the distance of $t$ to the nearest integer (so $\|a/q\| = a/q$ if $1\le a\le (q-1)/2$ and $\|a/q\| = (q-a)/q$ if $(q+1)/2\le a<q$). We also have the trivial bound $|\widehat{I_i}(a)|\le h_i$, and so
\[ q^{-1} \sum_{a_i\in \FF_q}|\widehat{I_i}(-a_i)| \le q^{-1}\sum_{|a|\le \frac{q}{2h_i}} h_i + 2q^{-1} \sum_{ \frac{q}{2h_i} < a \le \frac{q-1}{2}} \frac{q}{2a} \le \frac{h_i}{q}+ 1+\sum_{a=\lfloor q/(2h_i)\rfloor+1}^{(q-1)/2}\frac{1}{a}.\]
From the bounds  $\sum_{a=m_1+1}^{m_2}\tfrac{1}{a} \le \int_{m_1}^{m_2}\tfrac{ \mathrm{d}t}{t} = \log(m_2/m_1)$  for $m_2\ge m_1 \ge 1$ and $\sum_{a=1}^{m_2}\tfrac{1}{a} \le \log(m_2)+1$ we deduce that 
\begin{equation}\label{eq:FourierL1}
q^{-1}\sum_{a_i\in \FF_q}|\widehat{I_i}(-a_i)|  \le \log(h_i)+3.
\end{equation}
Moreover, by Cauchy--Schwarz,
\[ q^{-1} \sum_{a_i \in \FF_q} |\widehat{I_i}(-a_i)| \le q^{-\frac{1}{2}} \sqrt{\sum_{a_i \in \FF_q} |\widehat{I_i}(a_i)|^2} = \sqrt{h_i} \]
where the last equality is Parseval for functions on $\FF_q$. We have shown $q^{-1}\sum_{a_i\in \FF_q}|\widehat{I_i}(-a_i)|$ is bounded by $\min\{\sqrt{h_i},\log(h_i)+3\}$, implying that $M_2(\mathbf{I}) \le \prod_{i=1}^{k} h_i^{-1}  \min\{ \sqrt{h_i}, \log (h_i)+3\}$, which concludes the proof.
\begin{remark}
If some $h_i$ is larger than $q/2$ one can gain a minor improvement by passing from $I_i$ to its complement $I_i^c = \FF_q \setminus I_i$, which is an interval of length $q-h_i$. Because $\widehat{I_i}(a)=-\widehat{I_i^c}(a)$ as long as $a\neq 0$, this allows replacing the bound $\min\{\sqrt{h_i},\log(h_i)+3\}$ on $q^{-1}\sum_{a_i\in \FF_q}|\widehat{I_i}(-a_i)|$ by $1+\min\{\sqrt{q-h_i},\log(q-h_i)+3\}$; the extra $1$ comes from $q^{-1}(\widehat{I_i}(0)-\widehat{I_i^c}(0))=q^{-1}(2h_i-q) \le 1$.
\end{remark}
\subsection{Proof of Theorem~\ref{thm:diag}}
Let $A_i = E_{i,i}\in \Mn$ be the elementary matrix with $1$ in the $(i,i)$-th entry and $0$ otherwise.
Let us take the functionals $L_i(g)=\Tr(A_ig)=g_{i,i}$ ($1\le i \le n$) and $I_i=\{ x_i\}$, $x_i \in \FF_q$. Then, by \eqref{eq:Fourierinput},
\[	S((g_{i,i})_{i=1}^{n}, (\{x_i\})_{i=1}^{n})  =q^{\binom{n}{2}-n}\sum_{ (a_1,\ldots,a_n) \in \FF^n_q} \psi(-\sum_{i=1}^{n}a_i x_i) (-1)^{\rank(\sum_{j=1}^{n} A_j a_j)} \prod_{i=1}^{n-\rank(\sum_{j=1}^{n} A_j a_j)}(q^i-1). \]
The rank of $\sum_{j=1}^{n} A_j a_j$ is equal to $| \{ 1 \le i \le n: a_i \neq 0\}|$, so  
\begin{equation}\label{eq:Sgii} 	S((g_{i,i})_{i=1}^{n}, (\{x_i\})_{i=1}^{n})  =q^{\binom{n}{2}-n}\sum_{r=0}^{n} (-1)^r \prod_{i=1}^{n-r}(q^i-1) F((x_i)_{i=1}^{n},r)
\end{equation}
where
\[ F((x_i)_{i=1}^{n},r) := \sum_{ \substack{(a_1,\ldots,a_n) \in \FF^n_q:\\ |\{1 \le i \le n:\, a_i \neq 0\}|=r}} \psi(-\sum_{i=1}^{n}a_i x_i).\]
We shall evaluate $F$. Denoting by $U$ the set of indices $u$ with $a_u \neq 0$, we have
\[ F((x_i)_{i=1}^{n},r) = \sum_{\substack{U \subseteq \{1,2,\ldots,n\} \\ |U|=r}} G((x_u)_{u \in U})\]
where
\begin{multline*}G((x_u)_{u\in U}) := \sum_{ \substack{(a_u)_{u \in U} \in \FF^U_q:\\ \forall u \in U: \, a_u \neq 0}} \psi(-\sum_{u\in U} a_u x_u) = \sum_{ (a_u)_{u \in U} \in \FF^U_q} \psi(-\sum_{u\in U} a_u x_u) \prod_{u\in U}(1-\mathbf{1}_{a_u=0}) \\= \sum_{T \subseteq U} (-1)^{|T|} H((x_u)_{u \in U \setminus T})
\end{multline*}
for
\[ H((x_u)_{u \in U \setminus T}) := \sum_{(a_u)_{u \in U \setminus T} \in \FF_q^{U\setminus T}} \psi(-\sum_{u \in U\setminus T} a_u x_u) = q^{|U\setminus T|} \mathbf{1}_{\forall u \in U\setminus T:\, x_u =0}.\]
Let $M:=\{1 \le i \le n : x_i =0 \}$. It follows that  
\begin{align*}
    F((x_i)_{i=1}^{n},r) &=  \sum_{\substack{T\subseteq U \subseteq \{1,2,\ldots,n\} \\ |U|=r}} (-1)^{|T|} q^{|U\setminus T|} \mathbf{1}_{\forall u \in U\setminus T:\, x_u=0} \\
    &=  \sum_{\substack{T\subseteq U \subseteq \{1,2,\ldots,n\} \\ |U|=r \\ U \setminus T \subseteq M}} (-1)^{|T|} q^{|U\setminus T|} \\
    &= \sum_{j=0}^{r} (-1)^{r-j} q^{j} \binom{|M|}{j} \binom{n-j}{r-j}\\
    &=\binom{n}{r}(-1)^r \sum_{j=0}^{\min\{|M|,r\}} (-q)^{j} \binom{r}{j} \frac{(|M|)_j}{(n)_j} 
    \end{align*} 
    where $j$ stands for the size of $U\setminus T$. Plugging this in \eqref{eq:Sgii} we obtain the desired result.

\begin{remark}
The function $F((x_i)_{i=1}^{n},r)$ appeared in the coding theory literature, where it is shown that it agrees with the $q$-ary Krawtchouk polynomial \cite[Theorem~4.1]{Delsarte73}. It arises as an eigenvalue of a certain adjacency matrix \cite[Theorem~4.2]{Delsarte73}.
\end{remark}
\subsection{Proof of Corollary~\ref{cor:diag}}
We first show the total variation distance in \eqref{eq:TVdiag} is $O(nq^{-n})$, given \eqref{eq:diagpw}. Summing the error term in \eqref{eq:diagpw} over all possible choices of $(x_i)_{i=1}^{n}$, we obtain
\begin{multline*}
d_{\TV}( (g_{i,i})_{i=1}^{n},(U_i)_{i=1}^{n})\ll \sum_{m=0}^{n}\sum_{\substack{(x_i)_{i=1}^{n} \in \FF_q^n:\\ |\{1\le i \le n: \,x_i= 0\}| = m}} q^{-2n}(qm+n) \\
=q^{-2n}\sum_{m=0}^{n}\binom{n}{m}(q-1)^{n-m}(qm+n)=q^{-2n}(n q^n+nq^n)=2nq^{-n}
\end{multline*}
using the binomial theorem. If $n\le 4$, \eqref{eq:TVdiag} is established since $n\ll \sqrt{nq}$. Otherwise, we apply Lemma~\ref{lem:dtv}: defining $f(r) := |\{ (a_j)_{j=1}^{n} \in \FF_q^n : \rank(\sum_{j=1}^{n} E_{j,j}a_j)=r\}| = (q-1)^r \binom{n}{r}$ we have
\[d_{\TV}( (g_{i,i})_{i=1}^{n},(U_i)_{i=1}^{n})\ll \sqrt{\sum_{r=1}^{n} f(r) q^{2\left(-nr+\binom{r}{2}\right)}}.\]
The sequence $b_r := f(r)q^{2\left(-nr+\binom{r}{2}\right)}$ satisfies
\[ \frac{b_{r+1}}{b_r} =\frac{n-r}{r+1} (q-1) q^{-2(n-r)} \le nq^{1-n} \le \frac{1}{2}\]
for $n \ge 4$, $q\ge 2$ and $1\le r \le n/2$. Moreover, $\sum_{r=0}^{n}f(r) =q^n$ holds, and $q^{2\left(-nr+\binom{r}{2}\right)}$ is decreasing in $1\le r \le n$. Thus
\[d_{\TV}( (g_{i,i})_{i=1}^{n},(U_i)_{i=1}^{n})\ll \sqrt{b_1} +  \sqrt{\sum_{r=0}^n f(r)}\cdot \max_{n\ge r>n/2} \sqrt{q^{-2nr+2\binom{r}{2}}}  \ll \frac{\sqrt{qn}}{q^n}\]
and \eqref{eq:TVdiag} is established. It remains to demonstrate \eqref{eq:diagpw}. Recall $|\GLn|=q^{\binom{n}{2}}\prod_{i=1}^{n}(q^i-1)$. The contribution of $r=0$ to the right-hand side of \eqref{eq:Sgiiform} is exactly $q^{-n}|\GLn|$, so using the triangle inequality, \eqref{eq:Sgiiform} implies that 
\begin{equation}\label{eq:Sgiiform2}
S((g_{i,i})_{i=1}^{n}, (\{x_i\})_{i=1}^{n})- q^{-n}|\GLn|  \ll q^{-n}|\GLn|\sum_{r=1}^{n} a_r
\end{equation}
where
\begin{equation}\label{eq:ardef}
a_r= \binom{n}{r}\prod_{i=n-r+1}^{n}(q^i-1)^{-1}   \sum_{j=0}^{\min\{m,r\}} q^j \binom{r}{j}\frac{(m)_j}{(n)_j}
\end{equation}
and $m=|\{ 1 \le i \le n: x_i = 0\}|$. 
We have 
\[ a_1 =  \frac{n}{q^n-1} \big(1+q\frac{m}{n}\big) \ll q^{-n}(qm+n).\]
To conclude it suffices to demonstrate that $\sum_{i=1}^{n} a_i \ll a_1$. We have
\[ \frac{a_{i+1}}{a_i} = \frac{n-i}{i+1}(q^{n-i}-1)^{-1} \frac{g_{i+1}}{g_i}\]
for $g_i = \sum_{j=0}^{\min\{m,i\}} q^j \binom{i}{j}\frac{(m)_j}{(n)_j}$. Observe that if $i \ge m$ then
\[ \frac{g_{i+1}}{g_i} \le \max_{0 \le j \le m} \frac{\binom{i+1}{j}}{\binom{i}{j}} = \frac{i+1}{i+1-m} \le i+1.\]
Similarly, if $1\le i\le m-1$ then
\[ \frac{g_{i+1}}{g_i} =\frac{q^{i+1} \frac{(m)_{i+1}}{(n)_{i+1}}}{g_i} + \frac{\sum_{j=0}^{i} q^j \binom{i+1}{j}\frac{(m)_j}{(n)_j}}{\sum_{j=0}^{i} q^j \binom{i}{j}\frac{(m)_j}{(n)_j}}  \le \frac{q^{i+1}\frac{(m)_{i+1}}{(n)_{i+1}}}{q^i \frac{(m)_i}{(n)_i}}+\max_{0 \le j \le i} \frac{\binom{i+1}{j}}{\binom{i}{j}} = q \frac{m-i}{n-i} + i+1.\]
In particular, $g_{i+1}/g_i = O(i+q)$ holds uniformly for $1 \le i \le n-1$ which implies that \[\frac{a_{i+1}}{a_i} \ll \frac{n-i}{q^{n-i}} + \frac{n-i}{iq^{n-i-1}}\]
holds for $1\le i \le n-1$. In particular, $a_{i+1}/a_i = O(1)$ for $1 \le i \le n-1$ and $a_{i+1}/a_i \le 1/2$ if $n-i \ge C$ for some absolute constant $C>0$. This yields the estimate $\sum_{i=1}^{n} a_i \ll a_1$ as needed.

\begin{remark}\label{rem:bin}
When deriving \eqref{eq:Sgiiform2} via the triangle inequality above we replaced $(-q)^j$ with $q^j$ in \eqref{eq:Sgiiform}. If instead we keep the sign in \eqref{eq:Sgiiform}, we obtain in the same way
\begin{equation}\label{eq:Sgiiform3}
	S((g_{i,i})_{i=1}^{n}, (\{x_i\})_{i=1}^{n})- q^{-n}|\GLn|  \ll q^{-n}|\GLn|( a'_1 +a'_2+ \sum_{r=3}^{n} a_r)
\end{equation}
where $a_r$ is as in \eqref{eq:ardef} if $r\le n$, $a_r=0$ if $r>n$,
\[ a'_r= \binom{n}{r}\prod_{i=n-r+1}^{n}(q^i-1)^{-1}  \left|\sum_{j=0}^{\min\{m,r\}} (-q)^j \binom{r}{j}\frac{(m)_j}{(n)_j}\right|\]
if $r\le n$, and $a'_r=0$ if $r>n$. For $r=1$, $a'_1 \asymp q^{-n} |n-qm|$. The argument proving Corollary~\ref{cor:diag} already shows $\sum_{r=3}^{n}a_r \ll a_3$. A computation reveals $a'_2+a_3 \ll q^{2-2n}n+a'_1$ and we obtain
\begin{equation}\label{eq:diagpw2} S((g_{i,i})_{i=1}^{n}, (\{x_i\})_{i=1}^{n})= q^{-n}|\GLn| (1+O(q^{-n}|n-qm|+q^{2-2n}n \mathbf{1}_{n\ge 2})).
\end{equation}
\end{remark}
\subsection{Proof of Theorem~\ref{thm:submatrices}}
Let $1 \le m_1,m_2 \le n$. By \eqref{eq:trivial}, the estimate is trivial if $m_1+m_2 >n$. Suppose $m_1+m_2 \le n$. Let $E_{i,j}\in \Mn$ be the elementary matrix with $1$ in the $(i,j)$-th entry.  
By Lemma~\ref{lem:gensingle} with $(L_i(g))_{i=1}^{m_1m_2} = (g_{i,j})_{1\le i \le m_1,\,1\le j \le m_2}$, we find that (since $g_{i,j}=\Tr(E_{j,i}g)$)
\begin{equation}\label{eq:entrypt} S(  (g_{i,j})_{\substack{1\le i \le m_1 \\ 1 \le j \le m_2}}, (\{x_{i,j}\})_{\substack{1 \le i \le m_1 \\ 1 \le j \le m_2}})=q^{-m_1 m_2 }|\GLn| \bigg( 1+O\bigg( \sum_{r=1}^{n} f(r) q^{-nr+\binom{r}{2}}\bigg)\bigg) 
\end{equation}
for any $x_{i,j} \in \FF_q$, where
\[ f(r) := |\{(a_{i,j})_{1 \le i \le m_1,\, 1\le j \le m_2}  \in \FF_q^{m_1 m_2}: \rank\big(\sum_{1 \le i \le m_1,\,1 \le j \le m_2} a_{i,j} E_{j,i}\big)=r\}|.\]
The function $f(r)$ counts the number of $m_1 \times m_2$ matrices over $\FF_q$ of rank $r$, a count which was evaluated by Landsberg \cite{Landsberg}. It is $0$ if $r > \min\{m_1,m_2\}$, and otherwise it is 
\begin{equation}\label{eq:lands}
f(r) = \prod_{i=0}^{r-1} \frac{(q^{m_1}-q^i)(q^{m_2}-q^i)}{q^{r}-q^i} \ll q^{(m_1+m_2-r)r}.
\end{equation}
If $m_1 + m_2 \le n$, it follows that the error term in \eqref{eq:entrypt} is 
\[ \sum_{r=1}^{n} f(r) q^{-nr+\binom{r}{2}}\ll \sum_{r=1}^{\min\{m_1,m_2\}} q^{(m_1+m_2-r)r-nr+\binom{r}{2}}.\]
Observe $r\mapsto (m_1+m_2-r)r-nr+\binom{r}{2} = (m_1+m_2-n-\frac{1}{2}-\frac{r}{2})r $ is a decreasing integer-valued function in $1\le r \le \min\{m_1,m_2\}$, proving the error term in \eqref{eq:entrypt} is $\ll q^{m_1+m_2-1-n}$. 
\subsection{Proof of Lemma~\ref{lem:form}}
Recall $m_1\le m_2$ is assumed. Let $X=(x_{i,j})_{1\le i\le m_1,\, 1\le j \le m_2} \in \mathrm{M}_{m_1 \times m_2}(\FF_q)$, and recall the rank of $X$ is $r$. 
Evaluating 
\[S(  (g_{i,j})_{\substack{1\le i \le m_1 \\ 1 \le j \le m_2}}, (\{x_{i,j}\})_{\substack{1 \le i \le m_1 \\ 1 \le j \le m_2}})\]
is the same as counting matrices 
\[Y_1 \in \mathrm{M}_{m_1 \times (n-m_2)}(\FF_q),\qquad Y_2 \in \mathrm{M}_{(n-m_1) \times m_2}(\FF_q), \qquad Y_3 \in \mathrm{M}_{(n-m_1)\times (n-m_2)}(\FF_q)\]
such that
\[ g:=\begin{pmatrix} X && Y_1 \\ Y_2 &&Y_3 \end{pmatrix} \in \GLn.\]
We write $X$ as
\[ X = X_1 D_{r} X_2\]
where $X_i \in \mathrm{GL}_{m_i}(\FF_q)$ ($i=1,2$) and $D_r \in \mathrm{M}_{m_1\times m_2}(\FF_q)$ is supported on the diagonal, which consists of $r$ ones followed by zeros. Then $g\in \GLn$ holds if and only if 
\[ \begin{pmatrix} X^{-1}_1 && 0 \\ 0 && I_{n-m_1} \end{pmatrix}g \begin{pmatrix} X^{-1}_2 && 0 \\ 0 &&I_{n-m_2} \end{pmatrix} \in \GLn\]
holds. Let $Y'_1 := X^{-1}_1 Y_1$ and $Y'_2:= Y_2 X^{-1}_2$. Then simplifying the above product of three matrices we see that we are counting matrices $Y'_1,Y'_2,Y_3$ such that
\[ g':=\begin{pmatrix} D_r && Y'_1 \\ Y'_2 &&Y_3 \end{pmatrix}\]
is invertible. We first select $Y'_2$: as the first $m_2$ columns of $g'$ must be linearly independent, we select the first $r$ columns of $Y'_2$ arbitrarily; there are $q^{(n-m_1)r}$ choices. The next $m_2-r$ columns of $Y'_2$ are chosen to be linearly independent vectors in $\FF_q^{n-m_1}$; this is possible only if $m_2-r \le n-m_1$ holds, and there are $\prod_{i=0}^{m_2-r-1}(q^{n-m_1}-q^i)$ choices. It remains to choose $Y'_1,Y'_3$ in such a way that the $j$th column of $g'$ is linearly independent of the earlier ones for every $n\ge j>m_2$. Proceeding column by column we find there are $\prod_{i=m_2}^{n-1}(q^n-q^i)$ choices for these columns. This gives the desired result.
\subsection{Proof of Theorem~\ref{thm:submatricestv}}
If $X,Y$ are two  random variables taking values in a finite set $\Omega$, we have \cite[p.~67]{ABR}
\[ d_{\TV}(X,Y) = \sum_{x \in \Omega'} (\PP(X=x) - \PP(Y=x))\]
where $\Omega' := \{ x \in \Omega: \PP(X=x)\ge \PP(Y=x)\}$.
We apply this equality with
\[ X = (g_{i,j})_{\substack{1\le i\le m_1 \\ 1 \le j \le m_2}},\qquad  Y=(U_{i,j})_{\substack{1 \le i \le m_1 \\ 1 \le j \le m_2}}.\]
Suppose without loss of generality that $m_1 \le m_2$. Then, for every $(x_{i,j})_{1\le i \le m_1,\, 1\le j \le m_2} \in \mathrm{M}_{m_1\times m_2}(\FF_q)$,
\begin{equation}\label{eq:Pyprov}
\PP(Y=(x_{i,j})_{\substack{1\le i\le m_1 \\ 1\le j\le m_2}}) = q^{-m_1 m_2}
\end{equation}
and, by Lemma~\ref{lem:form},
\begin{equation}\label{eq:Pxprov}
\PP(X=(x_{i,j})_{\substack{1\le i\le m_1 \\ 1\le j\le m_2}}) = q^{-m_1 m_2} \frac{\prod_{j=n-m_1-m_2+r+1}^{n-m_1}(1-q^{-j})}{\prod_{j=n-m_2+1}^{n}(1-q^{-j})}=: q^{-m_1 m_2} h(r)
\end{equation}
where $r$ is the rank of $(x_{i,j})_{1\le i \le m_1,\, 1\le j \le m_2} \in \mathrm{M}_{m_1\times m_2}(\FF_q)$ and
\[ h(r)=\frac{\prod_{j=n-m_1-m_2+r+1}^{n-m_1}(1-q^{-j})}{\prod_{j=n-m_2+1}^{n}(1-q^{-j})}. \]
Observe $h(r)=0$ if $0\le r\le m_1+m_2-1-n$, while for $m_1+m_2-n\le r \le m_1$, $h$ is positive and increasing. Moreover,
\[ h(m_1)=\frac{\prod_{j=n-m_2+1}^{n-m_1}(1-q^{-j})}{\prod_{j=n-m_2+1}^{n}(1-q^{-j})}= \prod_{j=n-m_1+1}^{n}(1-q^{-j})^{-1}>1\]
while
\[ h(m_1-1) = \frac{\prod_{j=n-m_2}^{n-m_1}(1-q^{-j})}{\prod_{j=n-m_2+1}^{n}(1-q^{-j})}=(1-q^{-(n-m_2)}) \prod_{j=n-m_1+1}^{n}(1-q^{-j})^{-1}<1\]
where in the last inequality we used $\prod_{j=n-m_1+1}^{n}(1-a_j)\ge 1-\sum_{j=n-m_1+1}^{n} a_j$. 
Thus $\Omega'$ above is the set of matrices in $\mathrm{M}_{m_1 \times m_2}(\FF_q)$ of rank  $m_1$. It follows that the total variation distance in Theorem~\ref{thm:submatricestv} is precisely
\[  q^{-m_1m_2}(h(m_1)-1) \sum_{\substack{A \in \mathrm{M}_{m_1\times m_2}(\FF_q)\\\rank(A)=m_1}}1. \]
We have $h(m_1)-1 \asymp q^{-(n-m_1+1)}$ using $\log(1+x) \asymp x$ for $0\le x \le 1$. The proof is concluded via Landsberg's result  \cite{Landsberg}, already used when proving Theorem~\ref{thm:submatrices}, saying the $A$-sum is $\prod_{i=0}^{m_1-1} (q^{m_2}-q^i) \asymp q^{m_1 m_2}$.
\begin{remark}\label{rem:alt}
Suppose $m_1 \le m_2$. In the notation of Theorem~\ref{thm:submatrices}, formulas \eqref{eq:Pyprov}--\eqref{eq:Pxprov} above give
\[S(  (g_{i,j})_{\substack{1\le i \le m_1 \\ 1 \le j \le m_2}}, (\{x_{i,j}\})_{\substack{1 \le i \le m_1 \\ 1 \le j \le m_2}})=q^{-m_1 m_2 }|\GLn| ( 1+O( |h(r)-1|))\]
and this error is optimal, where $r$ is the rank of $(x_{i,j})_{1\le i \le m_1,\, 1\le j \le m_2} \in \mathrm{M}_{m_1\times m_2}(\FF_q)$. If $0\le r\le m_1+m_2-1-n$ then $h(r)=0$ and this recovers Theorem~\ref{thm:submatrices}. If $r=m_1$ then $|h(m_1)-1|=h(m_1)-1\asymp q^{-(n-m_1+1)}$, which improves Theorem~\ref{thm:submatrices} by a factor of size $q^{\min\{m_2,n-m_1+1\}}$. If $\max\{0,m_1+m_2-n\}\le r<m_1$ then $|h(r)-1|=1-h(r) \ll q^{-(n-m_1-m_2+r+1)}$ by using $-\log(1-x)=x+O(x^2)$ for $0\le x\le 1/2$; this improves Theorem~\ref{thm:submatrices} by a factor of size $q^{\min\{r,n-m_1-m_2+r+1\}}$.
\end{remark}
\subsection{Proof of Theorem~\ref{thm:pvfunc}}
Our starting point is Lemma~\ref{lem:gen}. We specialize the functionals there to $(L_i(g))_{i=1}^{k} = (\Tr( E_{j,i} g))_{1\le i,j\le n}$. The lemma reduces the proof to the bound
\begin{equation}\label{eq:reduce} \sum_{r=1}^{n} q^{-nr+\binom{r}{2}}\sum_{\substack{ A=(a_{i,j})_{1\le i,j\le n} \in \Mn \\ \rank A =r}} \prod_{1\le i,j\le n} \frac{|\widehat{I_{i,j}}(-a_{i,j})|}{\widehat{I_{i,j}}(0)} \ll \left( 3\frac{q^{\frac{n-1}{2n-1}}(\log (h)+3)}{h}\right)^{2n-1}.
\end{equation}
We shall show the contribution of matrices $A$ of rank $r$ to the above sum is 
\begin{equation}\label{eq:rankr}
\sum_{\substack{ A=(a_{i,j})_{1\le i,j\le n} \in \Mn \\ \rank A =r}} \prod_{1\le i,j\le n} \frac{|\widehat{I_{i,j}}(-a_{i,j})|}{\widehat{I_{i,j}}(0)} \ll \binom{n}{r}^2 M^{r(2n-r)}
\end{equation}
where $M:=\sum_{x\in \FF_q}|\widehat{I_{i,j}}(x)|/\widehat{I_{i,j}}(0)$; observe $M$ does not depend on $i,j$ because all the $I_{i,j}$ are intervals of length $h$. Moreover, $M\le q (\log(h)+3)/h$ by \eqref{eq:FourierL1}. By \eqref{eq:rankr}, we bound the left-hand side of \eqref{eq:reduce} by
\begin{equation}\label{eq:totalranks}
\ll  \sum_{r=1}^{n} q^{-nr+\binom{r}{2}} \binom{n}{r}^2 M^{r(2n-r)} \le  \sum_{r=1}^{n} \binom{n}{r}^2 \left( \frac{q^{\frac{1}{2}-\frac{1}{2(2n-r)}} (\log (h)+3)}{h}\right)^{r(2n-r)}
\end{equation}
using $M\le q(\log(h)+3)/h$. We consider two cases. When 
\begin{equation}\label{eq:largeh} \frac{q^{\frac{1}{2}-\frac{1}{2(2n-1)}} (\log (h)+3)}{h} \le 1,
\end{equation}
the right-hand side of \eqref{eq:totalranks} is at most
\[  \max_{1\le r\le n}\left( \frac{q^{\frac{1}{2}-\frac{1}{2(2n-r)}} (\log(h)+3)}{h}\right)^{r(2n-r)} \sum_{r=1}^{n} \binom{n}{r}^2 =\left( \frac{q^{\frac{n-1}{2n-1}}(\log(h)+3)}{h}\right)^{2n-1} \left(\binom{2n}{n}-1\right)\]
which implies Theorem~\ref{thm:pvfunc} when \eqref{eq:largeh} holds. When \eqref{eq:largeh} fails to hold we simply invoke \eqref{eq:trivial}.

It remains to establish \eqref{eq:rankr}. Any $A\in \Mn$ has  $\binom{n}{r}^2$ $r \times r$ minors, and if $\rank A =r$ then at least one of these minors is invertible. We fix an $r \times r$ minor and consider the contribution to \eqref{eq:rankr} of matrices with that minor being invertible. After permuting rows and columns, such matrices have the block form
\[\begin{pmatrix}
	U & V \\ W & W U^{-1}V
\end{pmatrix}\]
where $U\in \mathrm{GL}_r(\FF_q)$, i.e.~the lower $(n-r)\times (n-r)$ block is determined by the rest of the blocks, that have $n^2-(n-r)^2=r(2n-r)$ entries in total. Thus such matrices contribute at most $M^{r(2n-r)}$ to \eqref{eq:rankr}, establishing the required bound. Here we used that the weight $|\widehat{I_{i,j}}(-a_{i,j})|/\widehat{I_{i,j}}(0)$ is at most $1$.
\begin{remark}
For $r=n$, the bound \eqref{eq:rankr} should be sharp (because the left-hand side is precisely $M^{n^2}$ if $A$ ranges over all $\Mn$). Consequently, the most optimistic bound on \eqref{eq:reduce} is $q^{-n^2+\binom{n}{2}}M^{n^2} \le q^{\binom{n}{2}}((\log(h)+3)/h)^{n^2}$, which tends to $0$ as $h/ (q^{(n-1)/(2n)}(\log(q)+3))\to \infty$. This is the limit of our Fourier method, and is nearly reached by Theorem~\ref{thm:pvfunc}. 
\end{remark}
\subsection{Proof of Lemma~\ref{lem:genzero}}\label{sec:iden}
Let $S \subseteq \{1,2,\ldots,n\} \times \{1,2,\ldots,n\}$. We apply \eqref{eq:Fourierinput} with $k=|S|$, $(L_i(g))_{i=1}^{k} = (g_s)_{s \in S}$ and $I_i = \{0\}$ for all $i$. The left-hand side of \eqref{eq:Fourierinput} becomes $T_q(n,S)$ and we obtain that 
\[T_q(n,S) =
q^{\binom{n}{2}-|S|}\sum_{ (a_s)_{s\in S} \in \FF^S_q} (-1)^{\rank(\sum_{s \in S} a_s E^T_s )}\prod_{i=1}^{n-\rank(\sum_{s \in S} a_s E^T_s)}(q^i-1)\]
where $E_s\in \Mn$ is the elementary matrix with $1$ in the $s$-th entry and $0$ otherwise. We now rearrange the sum according to the value of $\rank(\sum_{s \in S} a_s E^T_s)$, which we denote by $r$.
\section{More on total variation}\label{sec:tv}
We pose the following problem.
\begin{prob}\label{prob}
Let $(U_i)_{i=1}^{k}$ be i.i.d.~random variables uniformly distributed on $\FF_q$. Let $g=g_n$ be a matrix chosen uniformly at random from $\GLn$. Determine the order of magnitude of 
\[ A_{n,k,q}:=\max_{L_1,\ldots,L_k} d_{\TV}( (L_i(g))_{i=1}^{k},(U_i)_{i=1}^{k})\]
where the maximum is taken over $k$-tuples of linearly independent functionals $L_i\colon  \Mn \to \FF_q$.
\end{prob}
We already saw $A_{n,k,q}\ll q^{\frac{k}{2}-n}$ in Theorem~\ref{thm:dtv}. In this section we record additional progress on the problem. We shall use the notation of Problem~\ref{prob}.
\begin{lem}\label{lem:tvtriv}
Uniformly in $1\le k \le n^2$, we have $A_{n,k,q}\ll 1/q$. For $k=n^2$ we have a matching lower bound.
\end{lem}
\begin{proof}
Let $m=m_n$ be a matrix chosen uniformly at random from $\Mn$. Let $E_{\mathrm{Inv}}$ be the event that $m$ is invertible. Note that $g$ and $m\mid E_{\mathrm{Inv}}$ have the same distribution. Moreover, $(L_i(m))_{i=1}^{k}$ has the same distribution as $(U_i)_{i=1}^{k}$. Thus
\[ d_{\TV}( (L_i(g))_{i=1}^{k},(U_i)_{i=1}^{k}) = d_{\TV}( (L_i(m))_{i=1}^{k}\mid E_{\mathrm{Inv}},(L_i(m))_{i=1}^{k}).\]
We now use the general fact that $d_{\TV}(X\mid E, X)\le 1-\PP(E)$ holds for any random variable $X$ and an event $E$ that happens with positive probability, see e.g.~\cite[Lem.~6.2]{Svante}. To conclude the upper bound, we observe $\PP(E_{\mathrm{Inv}})=|\GLn|/q^{n^2} = 1 +O(1/q)$ with an absolute implied constant. As for the lower bound, if $k=n^2$ then $(L_i(g))_{i=1}^{n^2}$ and  $(L_i(m))_{i=1}^{n^2}$, after a linear change of variables, are equal to $g$ and $m$, and so
\[ d_{\TV}( (L_i(g))_{i=1}^{k},(U_i)_{i=1}^{k}) =d_{\TV}(g,m)=  d_{\TV}(m\mid \{m \in E_{\mathrm{Inv}}\},m)=1-\PP(E_{\mathrm{Inv}})\asymp 1/q\]
where we used the general fact $d_{\TV}(X,X \mid \{ X \in  E\}) = 1-\PP(X\in E)$ \cite[Eq.~(7.25)]{Aldous} and the bound $1-\PP(E_{\mathrm{Inv}}) = 1-\prod_{i=1}^{n}(1-q^{-i})\ge 1-(1-q^{-1})=1/q$.
\end{proof}
For $1\le k \le n$ we can save a factor of $q^{-\frac{1}{2}}$ over Theorem~\ref{thm:dtv}.
\begin{thm}
For $1\le k \le n$, $A_{n,k,q} \ll q^{\frac{k}{2}-n-\frac{1}{2}}$.
\end{thm}
\begin{proof}
For $n=1$ this follows from Lemma~\ref{lem:tvtriv} so we assume $n\ge 2$. Let $L_1,\ldots,L_k$ be $k$ linearly independent functionals of the form $L_i(g)=\Tr(A_i g)$. We first suppose that 
\begin{equation}\label{eq:vdef}
V:=\Sp_{\FF_q}(A_1,\ldots,A_k)
\end{equation}
contains a matrix of rank at least $2$. By Lemma~\ref{lem:dtv}, it suffices to prove 
\[\sum_{r=1}^{n} f(r) q^{2\left(-nr+\binom{r}{2}\right)} \ll q^{k-2n-1}\]
with an absolute implied constant, where $f$ is defined in \eqref{eq:deff}. The contribution of $r\ge 2$ is $\ll q^{2(-2n+1)+k}\ll q^{k-2n-1}$ since $n\ge 2$, $f(r)\le q^k$ and $r\mapsto -nr+\binom{r}{2}$ is strictly decreasing. It then suffices to show $f(1) \ll q^{k-1}$; recall $f(1)$ is the number of rank-$1$ matrices in $V$. Since $V$ contains a matrix of rank at least $2$, there is a $2\times 2$ minor of $\sum_{i=1}^{k} a_i A_i$ that does not vanish as $a_i$ range over $\FF_q$, i.e.
\[ g(a_1,\ldots,a_k):=\det\begin{pmatrix}(\sum_{i=1}^{k} a_i A_i)_{i_1,j_1} & (\sum_{i=1}^{k} a_i A_i)_{i_1,j_2} \\ (\sum_{i=1}^{k} a_i A_i)_{i_2,j_1} & (\sum_{i=1}^{k} a_i A_i)_{i_2,j_2}\end{pmatrix}\]
is not identically zero for certain $1\le i_1,i_2,j_1,j_2\le n$.
If $\sum_{i=1}^{k} a_i A_i$ has rank $1$ then $g(a_1,\ldots,a_k)=0$, so $f(1)$ is bounded by the number of $\FF_q$-zeros of $g$. By the Schwartz--Zippel Lemma, since $g$ is a nonzero quadratic function we have $f(1) \le 2q^{k-1}$ as needed. 

It remains to consider the case where all the nonzero matrices in $V$ are of rank $1$. After a linear change of variables, we may assume without loss of generality that $A_i = E_{1,i}$ for $1\le i \le k$, say\footnote{Any rank-$1$ matrix is of the form $A_i=u_i v_i^T$ for some column vectors $u_i$ and $v_i$. If $\rank(A_i+A_j)=1$ then either $u_i,u_j$ are linearly dependent or $v_i,v_j$ are linearly dependent. Applying this to every pair of $A_i$ it follows that either the $u_i$ span a $1$-dimensional space or the $v_i$ do. Assume without loss of generality that we are in the first case, in which case after further relabeling we may assume that the $u_i$ are equal to a common $u$. It is then a matter of replacing $g$ by $g'=VgU$ for invertible $V,U\in \GLn$ with $e_i^T V=v_i^T$ and $Ue_1=u$: $g$ and $g'$ have the same distribution, while $\Tr(u_iv_i^T g) = \Tr(E_{1,i}g')$ for all $i$.}. In this case, Theorem~\ref{thm:submatricestv} with $m_1=k$ and $m_2=1$ gives the stronger bound $O(q^{-n})$ on the total variation distance.
\end{proof}
The main result of this section is the following bound on $A_{n,k,q}$.
\begin{thm}\label{thm:tv2}
	Uniformly in $n\le k \le n^2$ and $q$, we have $A_{n,k,q}\ll q^{\lfloor k/n\rfloor /2 -n/2 - 1/2}$.
\end{thm}
Before proving Theorem~\ref{thm:tv2} we derive a corollary.
\begin{cor}
For fixed $q$, $A_{n,k,q}$ tends to $0$ if and only if $(n^2-k)/n$ tends to $\infty$.
\end{cor}
\begin{proof}
If $(n^2-k)/n\to \infty$ we apply Theorem~\ref{thm:tv2} to conclude $A_{n,k,q}$ tends to $0$. Conversely, it suffices to exhibit for any fixed positive integer $R$ an example of $n(n-R)$ functionals $(L_i)_{i=1}^{k}$ where $ d_{\TV}( (L_i(g))_{i=1}^{k},(U_i)_{i=1}^{k})$ does not tend to $0$; indeed, for $k=n(n-R)$ we have $(n^2-k)/n = R$. For such an example, apply Theorem~\ref{thm:submatricestv} with $m_1=n$ and $m_2=n-R$, which produces an example of $k= n(n-R)$ functionals with  $ d_{\TV}( (L_i(g))_{i=1}^{k},(U_i)_{i=1}^{k})\asymp q^{-R-1}$. 
\end{proof}
By Lemma~\ref{lem:dtv}, in order to prove Theorem~\ref{thm:tv2} it suffices to prove 
\[\sum_{r=1}^{n} f(r) q^{2\left(-nr+\binom{r}{2}\right)} \ll q^{\lfloor k/n\rfloor -n-1}\]
holds with an absolute constant, where $f$ is defined in \eqref{eq:deff}. This in turn follows if we show
\[ f(1) \ll q^{\lfloor k/n\rfloor +n-1}\]
and, for $2\le r\le n$,
\[ f(r) \ll q^{rk/n+r(n-r)}.\]
This follows from the next two propositions. We set
$f(r;V)=\#\{A\in V:\rank(A)=r\}$ given $V\subseteq \Mn$.
\begin{prop}\cite{Schaathun}\label{prop:rank-one}
	For every subspace $V\subseteq \Mn$ of dimension $k$, \[f(1;V)\le \max_{\substack{n\ge a_1\ge \ldots \ge a_n \ge 0\\ a_1+\ldots+a_n=k}}\sum_{i=1}^{n}q^{i-1}(q^{a_i}-1).\]
\end{prop}
Schaathun \cite{Schaathun} does not state his result that way; instead, it is a result on the generalized Hamming weights of product codes, essentially bounding the intersection of the Segre variety $\Sigma_{n-1,n-1}$\footnote{It is the image of the embedding of $\mathbb{P}^{n-1}(\FF_q)\times \mathbb{P}^{n-1}(\FF_q)$ in $\mathbb{P}^{n^2-1}(\FF_q)$ defined by $(x,y)\mapsto xy^T$ where $x,y$ are column vectors; the image parametrizes rank-one matrices up to multiplication by scalars.} with $\mathbb{P}^{n-1}(V)$. The upper bound in Proposition~\ref{prop:rank-one} has a closed form. Given integers $n\ge 1$, $k\ge 0$, define
\[ 
B_n(k):=\frac{(q^a-1)(q^n-1)}{q-1}+q^a(q^b-1)\]
where $a,b$ are uniquely defined by $k=an+b$, $0\le b<n$.
\begin{lem}\label{lem:diagram}
	Suppose $1\le m\le n$ and $n\ge a_1\ge\cdots\ge a_m\ge0$ are
	integers with $\sum_{i=1}^{m} a_i=k$. Then
	\[ \sum_{i=1}^m q^{i-1}(q^{a_i}-1) \le B_n(k).\]
\end{lem}
Denote by $\binom{n}{r}_q$ the number of $r$-dimensional subspaces of $\FF_q^n$. It satisfies $\binom{n}{r}_q =\prod_{i=1}^{n}(q^i-1)/(\prod_{i=1}^{r}(q^i-1)\prod_{i=1}^{n-r}(q^i-1)) \ll q^{r(n-r)}$.
\begin{prop}\label{prop:rank-r}
	For every subspace $V\subseteq \Mn$ of dimension $k$, and every $1\le r\le n$, $f(r;V)\le \binom{n}{r}_q q^{rk/n}$.
\end{prop}
We are ready to prove Proposition~\ref{prop:rank-one}, Lemma~\ref{lem:diagram} and Proposition~\ref{prop:rank-r}. See \cite{Beelen2015,Beelen} for related bounds on $f$. 
\subsection{Proof of Proposition~\ref{prop:rank-one}}
	Let $V_i\subseteq V$ be the subspace of matrices whose first $i$ rows vanish, and let $R_i\subseteq\FF_q^n$ be the possible $i$th rows in $V_{i-1}$. The map $V_{i-1}\to R_i$ sending a matrix to its $i$th row is surjective, with kernel $V_i$. Hence
	$\dim R_i =\dim V_{i-1}-\dim V_i$, and summing over $i$ yields $\sum_{i=1}^n \dim R_i=k$. Let $d_i:=\dim R_i$.
	
	For a $1$-dimensional subspace $\ell\subseteq\FF_q^n$, fix a nonzero row vector $y_\ell \in \ell$, and put
	\[  X_\ell := \{x\in\FF_q^n:xy_\ell\in V\},  \qquad m_\ell := \#\{i:\ell\subseteq R_i\}.
	\]
	We now show 
	\begin{equation}\label{eq:dim}
		\dim X_\ell\le m_\ell. 
	\end{equation}
	For each $0\le i\le n$, define $X_{\ell,i} :=
	\{x\in X_\ell:x_1=\cdots=x_i=0\}$. Then
	$X_\ell = X_{\ell,0}\supseteq X_{\ell,1}
	\supseteq\cdots\supseteq X_{\ell,n}=\{0\}$. For each $1\le i\le n$, put
	$\delta_i := \dim X_{\ell,i-1}-\dim X_{\ell,i}$. Since $X_{\ell,i}$ is the kernel of the linear map
	$X_{\ell,i-1}\rightarrow\FF_q$ sending $x\mapsto x_i$, we have $\delta_i\in\{0,1\}$. Moreover, 
	$\dim X_\ell=\sum_{i=1}^n\delta_i$. We claim  $\delta_i \le \mathbf{1}_{\ell\subseteq R_i}$. It suffices to consider $\delta_i=1$. In this case there exists $x\in X_{\ell,i-1}$ such that $x_i\neq 0$. By definition,
	$x_1=\cdots=x_{i-1}=0$ and $xy_\ell\in V$. Thus the matrix $A=xy_\ell$ has its first $i-1$ rows equal
	to zero, so $A\in V_{i-1}$. Its $i$th row is $x_i y_\ell$,
	and hence $x_i y_\ell\in R_i$. Since $x_i\ne0$, we obtain $y_\ell\in R_i$, and therefore
	$\ell\subseteq R_i$. Consequently, $\delta_i\le \mathbf{1}_{\ell\subseteq R_i}$ and  \eqref{eq:dim} follows.
	
	There is a bijection between $X_{\ell}\setminus \{0\}$ and rank-one matrices in $V$ whose row space is $\ell$, given by $x\mapsto xy_{\ell}$. Hence 
	\[  f(1;V)=\sum_\ell(q^{\dim X_\ell}-1) \le\sum_\ell(q^{m_\ell}-1) \]
	by \eqref{eq:dim}. Reorder the $R_i$ so that $d_1\ge\cdots\ge d_n$.
	Expanding $q^{m_\ell}=\prod_i(1+(q-1)\mathbf1_{\ell\subseteq R_i})$
	and using $\dim\bigcap_{i\in I}R_i\le d_{\max I}$ gives
	\[ \sum_\ell(q^{m_\ell}-1)=\sum_{\varnothing\neq I\subseteq \{1,2,\ldots,n\}}
	(q-1)^{|I|-1}\bigl(q^{\dim\cap_{i\in I}R_i}-1\bigr)\le\sum_{j=1}^n q^{j-1}(q^{d_j}-1)\]
	by the binomial theorem. The result follows.
\subsection{Proof of Lemma~\ref{lem:diagram}}
Given a sequence $a_1\ge \cdots \ge a_m \ge 0$ we form a Young diagram $Y_{\mathbf{a}}$ with $m$ rows (some may be empty), where $a_i$ is the number of cells in its $i$th row. We equip $Y_{\mathbf{a}}$ with a weight $(q-1)q^{i+j-2}$ on cell $(i,j)$. The sum of weights in the $i$th row of $Y_{\mathbf{a}}$ is $q^{i-1}(q^{a_i}-1)$, and so the total sum of weights in $Y_{\mathbf{a}}$ agrees with $\sum_{i=1}^{m} q^{i-1}(q^{a_i}-1)$. We denote the sum of weights in  $Y_{\mathbf{a}}$ by $W(\mathbf{a})$. We wish to prove $W(\mathbf{a}) \le B_n(k)$ where $n,k$ are as in the lemma. We proceed by induction on $m+n$. The cases $m=1$ or $n=1$ are immediate, as well as $k=0$. We explain the induction step.
	
Given a diagram $Y_{\mathbf{a}}$ with exactly $h$ nonempty rows, deleting its first column leaves a diagram $Y_{\mathbf{a'}}$ with $k-h$ cells. Putting weights in it according to the aforementioned rule, we have the equality
\begin{equation}\label{eq:weightrecur}
W(\mathbf{a})=q^h-1+qW(\mathbf{a'}).
\end{equation}
Necessarily $h\le m \le n$; we may assume $h=m$. We first suppose $h\le n-1$. From \eqref{eq:weightrecur} and the induction hypothesis, $W(\mathbf{a})\le q^h-1+qB_{n-1}(k-h)$. A direct calculation gives
\[  B_n(k)-\big(q^h-1+qB_{n-1}(k-h)\big) = \begin{cases}
		(q^{a+b-h}-1)(q^h-q^{a+1}),&\text{if }h\le a+b,\\
		(q^{h-a-b}-1)(q^{n-1-h+2a+b}-q^{a+b}),&\text{if } h>a+b, 
\end{cases} \]
where recall $k=an+b$. To finish the induction when $h\le n-1$, we want the left-hand side to be nonnegative. If $h>a+b$, this identity is nonnegative since $h\le n-1$. When $h\le a+b$, this identity is nonnegative when $h\ge a+1$ or when $a+b=h$. From $k=an+b$ and $k\le hn$, $h\ge a$ must hold. If $h\ge a+1$ we are done, and otherwise $h=a$ and $b=0$, forcing $a+b=h$ and we are done as well. Finally we consider the case $h=m=n$. We transpose $Y_{\mathbf{a'}}$. The transposed diagram has the same sum of weights but less than $n$ rows, so $m+n$ has decreased and we may appeal to the induction hypothesis and conclude $W(\mathbf{a})\le q^n-1+qB_n(k-n)=B_n(k)$. 
\subsection{Proof of Proposition~\ref{prop:rank-r}}
	For an $r$-dimensional subspace $H\subseteq\FF_q^n$, set
	$V(H) := \{A\in V:\text{every row of $A$ belongs to $H$}\}$.
	Let $R_i$ and $V_i$ be linear spaces as in the proof of Proposition~\ref{prop:rank-one}. We claim
	\begin{equation}\label{eq:dimsum}
		\dim V(H)\le\sum_{i=1}^n\dim(R_i\cap H).
	\end{equation}
	Indeed, let $W_i:=V(H)\cap V_i$ ($0\le i \le n$) and define the linear map $T_i\colon W_{i-1}\to \FF_q^n$ ($1\le i \le n$) sending a matrix to its $i$th row ($1\le i \le n$). We have $\ker T_i = W_i$, and the image of $T_i$ lies in $R_i\cap H$. Thus $\dim W_{i-1}-\dim W_i\le \dim(R_i\cap H)$, yielding \eqref{eq:dimsum} upon summation over $i$.
	
	We need the following identity: for every subspace
	$R\subseteq\FF_q^n$,
	\begin{equation}\label{eq:average}
		\sum_{\dim H=r}q^{n\dim(R\cap H)}
		=\binom{n}{r}_q q^{r \dim R},
	\end{equation}
	where $\binom{n}{r}_q$ counts the $r$-dimensional subspaces of
	$\FF_q^n$, and $H$ ranges over $r$-dimensional subspaces of $\FF_q^n$. To prove it, let $d=\dim R$ and observe the number of $r$-dimensional subspaces $H\subseteq \FF_q^n$ satisfying $\dim(R\cap H)=j$ is
	\[ \binom{d}{j}_q \binom{n-d}{r-j}_q q^{(d-j)(r-j)}, \]
	by first choosing $R\cap H \subseteq R$, then the image of $H$ in
	$\FF_q^n/R$, and finally a linear map from this
	image into $R/(R\cap H)$. Consequently,
	\begin{align*}
		\sum_{\dim H=r} q^{n\dim(R\cap H)}=
		\sum_j
		\binom{d}{j}_q
		\binom{n-d}{r-j}_q
		q^{(d-j)(r-j)+nj}=
		q^{rd}
		\sum_j
		\binom{d}{j}_q
		\binom{n-d}{r-j}_q
		q^{j(n-d-r+j)}=	q^{rd}\binom{n}{r}_q,
	\end{align*}
	where the last equality is the $q$-Vandermonde identity. This is \eqref{eq:average}.
	
	Every rank-$r$ matrix in $V$ belongs to $V(H)$ for exactly one $H$, so \eqref{eq:dimsum}, H\"older's inequality and \eqref{eq:average} imply
	\[  f(r;V)\le\sum_{\dim H=r}|V(H)|\le\sum_{\dim H=r}\prod_{i=1}^nq^{\dim(R_i\cap H)}\le\prod_{i=1}^n
	\left(\sum_{\dim H=r}q^{n\dim(R_i\cap H)}\right)^{1/n} =\binom{n}{r}_q q^{rk/n},\]
	where in the last equality we used $\sum_{i=1}^{n} \dim R_i = k$ (shown in the proof of Proposition~\ref{prop:rank-one}).

\bibliographystyle{abbrv}
\bibliography{references}

\Addresses

\end{document}